\documentclass{gtpart}    

\title{Control of fixed points over discrete $p$-toral groups, and existence and uniqueness of linking systems}

\author{R\'emi Molinier}
\givenname{R\'emi}
\surname{Molinier}
\address{}
\email{remi.molinier@univ-grenoble-alpes.fr}
\urladdr{}


\arxivreference{arXiv:1511.04291}

\newtheorem{thm}{Theorem}[section]    
\newtheorem{lem}[thm]{Lemma}         
\newtheorem{prop}[thm]{Proposition}
\newtheorem{cor}[thm]{Corollary}
\newtheorem{fact}[thm]{}
\newtheorem{thm*}{Theorem}
\newtheorem{prop*}{Proposition}

\theoremstyle{definition}
\newtheorem{defi}[thm]{Definition}

\newtheorem{case}{Case}

\newtheorem{rem}[thm]{Remark}

\usepackage{fullpage}
\usepackage{enumerate}
\usepackage[all]{xy}

\numberwithin{equation}{subsection}

\newcommand{\Fc}{\mathcal{F}}
\newcommand{\Li}{\mathcal{L}}
\newcommand{\T}{\mathcal{T}}
\newcommand{\Bc}{\mathcal{B}}
\newcommand{\A}{\mathcal{A}}
\newcommand{\Cc}{\mathcal{C}}

\newcommand{\Sc}{\mathcal{S}}
\newcommand{\Rc}{\mathcal{R}}

\newcommand{\Zc}{\mathcal{Z}}
\newcommand{\Or}{\mathcal{O}}
\newcommand{\Qc}{\mathcal{Q}}
\newcommand{\Nc}{\mathcal{N}}
\newcommand{\Wc}{\mathcal{W}}

\newcommand{\limproj}[1]{\lim\limits_{\substack{\longleftarrow \\ #1}}}

\newcommand{\Hom}{\text{Hom}}
\newcommand{\Aut}{\text{Aut}}
\newcommand{\Out}{\text{Out}}
\newcommand{\Inn}{\text{Inn}}

\newcommand{\Iso}{\text{Iso}}
\newcommand{\Mor}{\text{Mor}}

\newcommand{\Id}{\text{Id}}
\newcommand{\im}{\text{Im}\,}

\newcommand{\Syl}{\text{Syl}}
\newcommand{\ord}{\text{ord}}
\newcommand{\Ab}{\text{$\mathcal{A}$b}}

\begin{document}

\begin{abstract}  
The existence and uniqueness of linking systems associated to saturated fusion systems over discrete $p$-toral groups was proved by Levi and Libman. Their proof makes indirect use of the classification of the finite simple groups. Here we extend some results and arguments of Glauberman and Lynd to show how to remove this assumption.  
\end{abstract}

\maketitle


The notion of $p$-local compact groups is due to Broto, Levi and Oliver \cite{BLO3}. It encodes in an algebraic way the $p$-local homotopy theory of compact Lie groups, $p$-compact groups, and some other families of similar nature. In particular, the theory of $p$-local compact groups includes the theory of $p$-local finite groups (see \cite{BLO2} or \cite[Chapter III]{AKO}). Specifically, a $p$-local compact group is a triple $(S,\Fc,\Li)$ where $S$ is a discrete $p$-toral group, $\Fc$ is a saturated fusion system over $S$, and $\Li$ is an associated centric linking system. One basic question about $p$-local compact groups is the following: given a saturated fusion system over a $p$-toral group $S$, is there an associated linking system, and if so, is this linking system unique?

In the theory of $p$-local finite groups, Chermak was the first to succeed in answering these questions. In \cite{Ch}, using the theory of partial groups and localities, he shows that, given a saturated fusion system over a finite $p$-group $S$, there exists an associated centric linking system which is unique up to isomorphism. His theory gives also a totally new and more group-theoretic approach to $p$-local finite groups. Later, Oliver \cite{O5} translated Chermak's proof into the language of obstruction theory, and then Levi and Libman \cite{LL} extended Oliver's proof to $p$-local compact groups. Thus, we have the following theorem.

\begin{thm*}[{\cite[Theorem B]{LL}}]\label{A}
Let $\Fc$ be a saturated fusion system over a discrete $p$-toral group. Then there exists a centric linking system $\Li$ associated to $\Fc$, and $\Li$ is unique up to isomorphism.
\end{thm*}

Chermak's proof of the existence and uniqueness of centric linking systems uses indirectly the classification of finite simple groups. In the same way, the proof of Oliver and the extension to $p$-local compact groups by Levi and Libman also rely on the classification of finite simple groups. Recently, Glauberman and Lynd \cite{GL} succeeded, using a result of Glauberman \cite{Gl71} on control of fixed points, in giving a classification-free proof of the $p$-local finite version of Theorem \ref{A}.

In this article, we present an extension of arguments of Glauberman and Lynd to the more general setup of $p$-local compact groups, and thus completing the process of removing the classification of finite simple groups from the proof of Theorem \ref{A}. 

Oliver's obstruction theory, and its extension to $p$-local compact groups by Levi and Libman, give a link between the existence and uniqueness of linking systems and a vanishing property of higher limits of a certain functor. For $\Fc$ a saturated fusion system over a discrete $p$-toral group $S$, let $\Or(\Fc^c)$ be the associated centric orbit category and let $\Zc:\Or(\Fc^c)^\text{op}\rightarrow \Ab$ denote the contravariant functor that associates to a subgroup its center. The existence and uniqueness of centric linking systems can be translated as follows in terms of obstruction theory.

\begin{prop*}[{\cite[Proposition 1.7]{LL}}]\label{00}
Let $\Fc$ be a saturated fusion system over a discrete $p$-toral group $S$. An associated centric linking system exists if $\limproj{\Or(\Fc^c)}{}^{\hspace{-0.5em}3} \Zc=0$, and if also $\limproj{\Or(\Fc^c)}{}^{\hspace{-0.5em}2} \Zc=0$, then that linking system is unique up to an isomorphism.
\end{prop*}

Using the ideas of Chermak, Oliver \cite{O5} in the $p$-local finite group case, and Levi and Libman \cite{LL} in the general case, proved the following Theorem which implies Theorem \ref{A}.

\begin{thm*}\label{0}
Let  $\Fc$ be a saturated fusion system over  a discrete $p$-toral group $S$. Then, for all $i>1$ if $p=2$, and all $i>0$ otherwise, 
\[\limproj{\Or(\Fc^c)}{}^{\hspace{-0.5em}i} \Zc=0.\]
\end{thm*}

\begin{proof}
This is \cite[Theorem B]{LL}.
When $p$ is odd, this also follows from the proof of \cite[Theorem 3.7]{LL} and Proposition \ref{oddcase} below. When $p=2$, it also follows from the proof of \cite[Theorem 3.7]{LL} and Proposition \ref{6.9} below.
\end{proof}

In the case of $p$-local finite groups, Glauberman and Lynd \cite{GL} gave a proof of Theorem \ref{0} which does not rely on the classification of finite simple groups. The key is to get a better understanding of control of fixed points of finite groups acting on abelian $p$-groups. They reduced the problem to the study of a small set of possible cases which derives from McLaughlin's classification of irreducible subgroups of $SL_n(2)$ generated by transvections \cite{McL69}. 
Here we extend the arguments in \cite{GL} to $p$-local compact groups in order to get a classification-free proof of Theorem \ref{0} and hence also of Theorem \ref{A}. 

In Section 1 we give some background on discrete $p$-toral groups and "offenders". 
Then we look at the case $p\neq 2$ in Section 2. This case is a direct consequence of \cite[Theorem A.1.4]{Gl71}, called the Norm Argument, extended to actions of a finite group on an abelian discrete $p$-toral group. 
The case $p=2$, developed in Section 3, is more delicate. We need first the notion (introduced in \cite{GL}) of "solitary offenders", in order to split into two cases. Either there is a minimal best offender which is not solitary or all minimal best offenders are solitary. We give the definitions in Subsection 3.1 and give some results which tell us when we can use the Norm Argument. In Subsection 3.2 we give the notion of $\Fc$-conjugacy functor which extends the notion of $\Gamma$-functor defined by Glauberman in \cite{Gl71} and which is also recalled in \cite[Appendix 2]{GL}. The principal results in the even case are given in the last two subsections.
 
\subsection*{Acknowledgments} 
I am very grateful to Andy Chermak for suggesting this problem and for fruitful discussions with him. I also want to thank Alex Gonzalez for useful discussions about $p$-local compact groups which have led to some simplifications in the arguments.  
  
\subsection*{Notation}
Conjugation maps and morphisms in a fusion system will be written on the right and in the exponent, while cocycle and cohomology classes for functors will be written on the left. For a group $G$ we denote by $\Sc_p(G)$ the set of all the $p$-subgroups of $G$. For $Y$ a subgroup of $G$, $\Sc_p(G)_{\geq Y}$ denotes the set of all the $p$-subgroups of $G$ which contain $Y$. Finally, $S_m$ will denote the symmetric group on the set $\{1,2,\dots,m\}$, and $SL_n(2)$ denotes the linear group of degree $n$ over $\mathbb{F}_2$, the field with two elements.
 
\section{Background}
\subsection{Virtual discrete \texorpdfstring{$p$}{Lg}-toral groups}
This section is drawn from \cite{LL}.
Let $p$ be a prime, to be fixed for the remainder of this paper.

\begin{defi}
Let $\Z/p^{\infty}$ denote the union of all $\Z/p^r$, $r\geq 1$ with the obvious inclusions.
\begin{enumerate}[(a)]
\item A \emph{discrete $p$-torus} is a group $T$ isomorphic to the direct product $(\Z/p^\infty)^r$ of $r$ copies of $\Z/p^{\infty}$, for some $r$ called the \emph{rank} of $T$.
\item A \emph{virtually discrete $p$-toral group} is a group $\Gamma$ which contains a normal discrete $p$-torus $T$ of finite index. If moreover the index is a power of $p$, we say that $\Gamma$ is a \emph{discrete $p$-toral group}.
\end{enumerate}
\end{defi}

In every virtually discrete $p$-toral group $\Gamma$, there is a unique maximal discrete $p$-torus. 
This subgroup, denoted $\Gamma_0$, is a characteristic subgroup of $\Gamma$ and is called the \emph{maximal torus} or the \emph{identity component} of $\Gamma$. A finite group is in particular a virtually discrete $p$-toral group in which $G_0$ is of rank 0, and it is a discrete $p$-toral group if and only if it is a $p$-group. Hence the theory of virtually discrete $p$-toral group includes the theory of finite groups.

The \emph{order} of a virtually discrete $p$-toral group $\Gamma$ is the pair 
\[\ord(\Gamma)=\left(rk(\Gamma_0),|\Gamma/\Gamma_0|\right)\]
with the left lexicographic order: $(m,n)<(m',n')$ if and only if $m<m'$ or, $m=m'$ and $n<n'$. 

Any subgroup $\Gamma'\leq \Gamma$ of a virtually discrete $p$-toral group is also a virtually discrete $p$-toral group. Moreover, $\ord(\Gamma')\leq \ord(\Gamma)$, with equality  if and only if $\Gamma=\Gamma'$.

A virtually discrete $p$-toral group $\Gamma$ contains a maximal normal discrete $p$-toral subgroup, denoted $O_p(\Gamma)$. It is the preimage in $\Gamma$ of the maximal normal $p$-subgroup $O_p(\Gamma/\Gamma_0)$ and, in particular, contains $\Gamma_0$.

We also have a notion of $p$-Sylow subgroup in a virtually discrete $p$-toral group. If $\Gamma$ is virtually discrete $p$-toral, $\Gamma$ contains a maximal discrete $p$-toral subgroup $S$ given by the preimage in $\Gamma$ of a Sylow $p$-subgroup of $\Gamma/\Gamma_0$.

\begin{defi}
Let $\Gamma$ be a virtually discrete $p$-toral group.
A \emph{Sylow $p$-subgroup} of $\Gamma$ is a maximal discrete $p$-toral subgroup $S$ of $\Gamma$. We denote by $\Syl_p(\Gamma)$ the collection of all the Sylow $p$-subgroups of $\Gamma$.
\end{defi}

The terminology of Sylow subgroup is motivated by the fact that any discrete $p$-toral subgroup of $\Gamma$ is conjugate to a subgroup of $S$. In particular, all the Sylow $p$-subgroups of $\Gamma$ are conjugate.

\begin{lem}[{\cite[Lemma 1.3]{LL}}]\label{FrattiniArg}
\begin{enumerate}[(a)]
\item For any pair $P,Q$ of discrete $p$-toral groups, if $P<Q$ then $P<N_Q(P)$.
\item (Frattini Argument) Suppose that $\Gamma'\unlhd\Gamma$ are virtually $p$-toral groups, and $S\in\Syl_p(\Gamma')$. Then $\Gamma=\Gamma'N_\Gamma(S)$.
\end{enumerate}
\end{lem}

\begin{defi}
Let $P$ be an abelian discrete $p$-toral group.
For every $n>0$, denote by $\Omega_n(P)$ the subgroup of $P$ consisting of those elements of $P$ whose order divides $p^n$.
\end{defi}

Notice that for every $n\geq 1$, $\Omega_n(P)$ is a characteristic subgroup of $P$.

\begin{lem}[{\cite[Lemma 2.1]{LL}}]\label{action with discreteptoral}
Let $D$ be an abelian discrete $p$-toral group.
Then, the following hold:
\begin{enumerate}[(a)]
\item $D\cong D_0\times E$ where $E$ is a finite abelian $p$-group. In particular, $D=\cup_{n=1}^\infty \Omega_n(D)$;
\item the only action of a discrete $p$-torus $T$ on $D$ is the trivial action;
\item if $D\unlhd \Gamma$ where $\Gamma$ is a virtually discrete $p$-toral group, then $\Gamma_0\leq C_\Gamma(D)$.
\end{enumerate}
\end{lem}

The following notion was introduced in \cite{Ch}. It plays a central role here (as it did also in \cite{O5} and \cite{LL}).

\begin{defi}
A \emph{general setup} is a triple $(\Gamma,S,Y)$ where
\begin{enumerate}
\item $\Gamma$ is a virtually discrete $p$-toral group, $S\in\Syl_p(\Gamma)$ and $Y\leq S$;
\item $\Gamma_0\leq Y\trianglelefteq\Gamma$ and $C_\Gamma(Y)\leq Y$.
\end{enumerate}
A \emph{reduced setup} is a general setup $(\Gamma,S,Y)$ which satisfies
\begin{enumerate}[(3)]
\item $Y=O_p(\Gamma)$, $C_S(Z(Y))=Y$, and $O_p(\Gamma/C_\Gamma(Z(Y)))=1$.
\end{enumerate}
\end{defi}

\subsection{Fusion systems over discrete \texorpdfstring{$p$}{Lg}-toral groups}

\begin{defi}\label{fusion system}
A \emph{fusion system} over a discrete $p$-toral group $S$ is a category whose objects are the subgroups of $S$ and whose morphisms are group monomorphisms such that
\begin{enumerate}[(i)]
\item $\Hom_S(P,Q)\subseteq \Hom_\Fc(P,Q)$ for all $P,Q\leq S$, where $\Hom_S(P,Q)$ denotes the set of homomorphisms induced by conjugations in $S$;
\item every morphism is the composite of an $\Fc$-isomorphism and an inclusion.
\end{enumerate} 

Two subgroups $P,Q\leq S$ are said to be \emph{$\Fc$-conjugate} if they are isomorphic as objects in $\Fc$, and we denote by $P^\Fc$ the $\Fc$-isomorphism class of $P$, commonly called the \emph{$\Fc$-conjugacy class} of $P$.
\end{defi}

\begin{defi}\label{fn,fc,receptive}
let $\Fc$ be a fusion system over a discrete $p$-toral group $S$ and let $P$ be a subgroup of $S$.
\begin{enumerate}[(a)]
\item $P$ is \emph{fully normalized} in $\Fc$ if, for every $Q\in P^\Fc$, $|N_S(Q)|\leq |N_S(P)|$.
\item $P$ is \emph{fully centralized} in $\Fc$ if, for every $Q\in P^\Fc$, $|C_S(Q)|\leq |C_S(P)|$.
\item $P$ is \emph{fully automized} in $\Fc$ if the group $\Out_\Fc(P):=\Aut_\Fc(P)/\Inn(P)$ is finite and, $\Out_S(P):=\Aut_S(P)/\Inn(P)$ is a Sylow $p$-subgroup of $\Out_\Fc(P)$.
\item $P$ is \emph{receptive} if, for every $Q\in P^\Fc$ and $\varphi\in\Hom_\Fc(Q,P)$, if we set
\[N_\varphi=\left\lbrace g\in N_S(Q)\mid \varphi^{-1} c_g \varphi\in\Aut_S(P)\right\rbrace,\]
then there is $\overline{\varphi}\in\Hom_\Fc(N_\varphi,N_S(P))$ such that $\overline{\varphi}|_Q=\varphi$. 
\end{enumerate}
Denote by $\Fc^f$ the set of all fully normalized subgroups of $S$.
\end{defi}

\begin{defi}\label{saturation}
A fusion system $\Fc$ over a discrete $p$-toral group $S$ is \emph{saturated} if
\begin{enumerate}[(I)]
\item Every $P\leq S$ which is fully normalized in $\Fc$ is fully centralized and fully automized ;\label{I}
\item Every $P\leq S$ which is fully centralized in $\Fc$ is receptive in $\Fc$; and\label{II}
\item If $P_1\leq P_2\leq \cdots \leq S$ is an increasing sequence of subgroups of $S$ and if $\varphi\in\Hom(\cup_{i\geq 1}P_i,S)$ is a homomorphism such that $\varphi|_{P_i}\in\Hom_\Fc(P_i,S)$ for all $i$, then $\varphi\in \Hom_\Fc(\cup_{i\geq 1}P_i,S)$.\label{III}
\end{enumerate}
\end{defi}

Notice that, if $\Fc$ is a saturated fusion system over a discrete $p$-toral group $S$, for every $P\leq S$, by axioms \eqref{I} and \eqref{II}, $\Out_\Fc(P)$ is also finite. 

The next Lemma is a classical result for fusion systems over finite $p$-groups which can easily be extended to fusion systems over discrete $p$-toral groups. It is a well known result, but since the author could not find a reference for it, a proof is provided here. 

\begin{lem}\label{saturation lem}
Let $\Fc$ be a saturated fusion system over a discrete $p$-toral subgroup $S$.
Let $P\leq S$. 
Assume that $P$ is fully automized and receptive. Then, for every $Q\in P^\Fc$, there exists $\varphi\in\Hom_\Fc(N_S(Q),N_S(P))$ such that $Q^\varphi=P$.
\end{lem}

\begin{proof}
Assume $P$ is fully automized and receptive and fix $Q\in P^\Fc$. Choose $\psi\in\Iso_\Fc(P,Q)$. Then $\Out_S(Q)^\psi$ is a $p$-subgroup of $\Out_\Fc(P)$ and hence, it is $\Out_\Fc(P)$-conjugate to a subgroup of $\Out_S(P)$ since $P$ is fully automized. Fix $\alpha\in\Aut_\Fc(P)$ such that $\Aut_S(Q)^{\psi\alpha}$ is contained in $\Aut_S(P)$. Then, $N_{\psi\alpha}=N_S(Q)$ and so $\psi\alpha$ extends to some homomorphism $\varphi\in\Hom_\Fc(N_S(Q),N_S(P))$ with, $Q^\varphi=Q^{\psi\alpha}=P$ and $\im(\varphi)\leq N_S(P)$.
\end{proof}

One basic family of saturated fusion systems is given by fusion systems of virtually discrete $p$-toral groups.
For $\Gamma$ a virtually discrete $p$-toral group, let $\Fc_S(\Gamma)$ be the fusion system over $S$ where, for every $P,Q\leq S$, $\Hom_{\Fc_S(\Gamma)}(P,Q)=\Hom_\Gamma(P,Q)$ is the set of monomorphisms induced by conjugation in $\Gamma$.

\begin{prop}[{\cite[Proposition 1.5]{LL}}]
Let $\Gamma$ be a virtually discrete $p$-toral group and $S\in\Syl_p(\Gamma)$. Then $\Fc_S(\Gamma)$ is saturated.
\end{prop}

\begin{defi}
Let $\Fc$ be a fusion system over a discrete $p$-toral group $S$.
A subgroup $P\leq S$ is \emph{$\Fc$-centric} if for every $Q\in P^\Fc$, $C_S(Q)=Z(Q)$.
Denote by $\Fc^c$ the collection of $\Fc$-centric subgroup of $S$. We also write $\Fc^c$ for the full subcategory of $\Fc$ with objects the $\Fc$-centric subgroups of $S$.
\end{defi}

The following lemma will be useful when working with well-placed subgroups with respect to a conjugacy functor (Subsection 3.2). The author is grateful to Alex Gonzalez for pointing this out.

\begin{lem}\label{P^Fc finite S-conjugacy}
Let $\Fc$ be a saturated fusion system over a discrete $p$-toral group $S$ and let $P$ be a subgroup of $S$.
Then $P^\Fc$ contains finitely many $S$-conjugacy classes.
\end{lem}

\begin{proof}
Lemma \ref{P^Fc finite S-conjugacy} is a direct corollary of \cite[Lemma 2.5]{BLO3} with $Q=S$ in that Lemma.
\end{proof}

\begin{defi}
Let $\Fc$ be a fusion system over a discrete $p$-toral group $S$.
A subgroup $P\leq S$ is \emph{$\Fc$-weakly closed} if $P^\Fc=\{P\}$.
\end{defi}

\subsection{Higher limits of the center functor}

\begin{defi}
Let $\Fc$ be a saturated fusion system over a discrete $p$-toral group $S$.
The \emph{orbit category}, denoted $\Or(\Fc)$, is the category where the objects are the subgroups of $S$, and with
 \[\Mor_{\Or(\Fc)}(P,Q)=\Hom_\Fc(P,Q)/\Inn(Q).\]
Denote by $\Or(\Fc^c)$ the full subcategory of $\Or(\Fc)$ whose objects are the $\Fc$-centric subgroups of $\Fc$.
\end{defi}

\begin{defi}
Let $\Fc$ be a fusion system over a discrete $p$-toral group $S$.
\begin{enumerate}[(a)]
\item A collection $\Rc\subseteq \Sc_p(S)$ is an \emph{interval} if $P\in \Rc$ whenever $P_1,P_2\in\Rc$ and $P_1\leq P\leq P_2$.
\item A collection $\Rc\subseteq \Sc_p(S)$ is called \emph{$\Fc$-invariant} if, for every $P\in\Rc$ and $\varphi\in\Hom_\Fc(P,S)$, $P^\varphi\in\Rc$. 
\end{enumerate}
\end{defi}

\begin{defi}
Let $\Fc$ be a saturated fusion system over a discrete $p$-toral group $S$.
Let $\Zc:\Or(\Fc^c)\rightarrow \Ab$ denote the functor sending $P$ to its center $Z(P)$.
\begin{enumerate}[(a)]
\item for an $\Fc$-invariant interval $\Rc$ in $\Fc^c$, we denote by $\Zc^\Rc_{\Fc^c}$ the functor $\Zc^\Rc_{\Fc^c}:\Or(\Fc^c)\rightarrow \Ab$ such that
\[
\Zc^\Rc_{\Fc^c}(P)=\left\lbrace
\begin{array}{ll}
\Zc(P)=Z(P) &\text{if }P\in\Rc\\
1 & \text{otherwise.}
\end{array}\right.
\]
\item If $\Rc$ is an $\Fc$-invariant interval in $\Fc^c$, set
\[L^*(\Fc;\Rc)={\limproj{\Or(\Fc^c)}}^*\;\Zc_{\Fc^c}^\Rc \]

\end{enumerate}
\end{defi}

These higher derived limits are cohomology groups of a cochain complex $C^*(\Or(\Fc^c),\Zc_{\Fc^c}^\Rc)$, in which $k$-cochains are maps on sequences of $k$-composable morphisms in the category. A $0$-cochain is a map sending $P\in\Fc^c$ to an element $u(P)\in\Zc_{\Fc^c}^\Rc(P)$, and a $1$-cochain is a map sending a morphism $\xymatrix{P\ar[r]^{[\varphi]}& Q}$ to an element $t([\varphi])\in\Zc_{\Fc^c}^\Rc(P)$. In our case, we will consider an $\Fc$-invariant collection of subgroups $\Qc$ which is closed under passing to overgroups. In this particular setting, if we write $\varphi^{-1}$ for the inverse of $\varphi:P\rightarrow P^\varphi$, the coboundary map on such $0$- and $1$-cochains are as follows:

\begin{equation}\label{0-chainrule}
du([\varphi])= 
u(Q)^{\varphi^{-1}}u(P)^{-1}\in \Zc_{\Fc^c}^\Rc(P)
\end{equation}
and
\begin{equation}\label{1-chainrule}
dt([\varphi][\psi])=t([\psi])^{\varphi^{-1}}t([\varphi\psi])^{-1}t([\varphi])\in \Zc_{\Fc^c}^\Rc(P)
\end{equation}

for any sequence $\xymatrix{P\ar[r]^\varphi & Q \ar[r]^\psi & R}$ of composable morphisms in $\Fc^c$. Notice that, if $P\notin \Qc$, then $du([\varphi])=1$ and $dt([\varphi][\psi])=1$.
We refer to \cite[Section III.5.1]{AKO} for more details.

\begin{lem}[{\cite[Lemma 1.15]{LL}}]\label{higherlimitswithS}
Let $\Fc$ be a saturated fusion system over a discrete $p$-toral group $S$. Assume that $\Qc$ is an $\Fc$-invariant interval in $\Fc^c$ such that $S\in\Qc$, and denote by $\Or(\Fc^\Qc)$ the full subcategory of $\Or(\Fc^c)$ with the set $\Qc$ as set of objects.
\begin{enumerate}[(a)]
\item Let $F:\Or(\Fc^c)^\text{op}\rightarrow\Ab$ be a functor such that $F(P)=0$ if $P\in\Fc^c\smallsetminus \Qc$. Let $F|_\Qc$ denote its restriction to $\Or(\Fc^\Qc)^\text{op}$. Then, 
\[{\limproj{\Or(\Fc^c)}}^*\;F\cong{\limproj{\Or(\Fc^\Qc)}}^*\;F|_\Qc\]
\item Suppose that $(\Gamma,S,Y)$ is a general setup and $\Fc=\Fc_S(\Gamma)$ and $\Qc=\Fc_{\geq Y}$. Then
\[L^k(\Fc;\Qc)\cong
\left\lbrace\begin{array}{cc}
Z(\Gamma) & \text{if k=0} \\
        0 & \text{if k>0}
\end{array}         
         \right.\]
\end{enumerate}
\end{lem}

\begin{lem}[{\cite[Lemma 1.16]{LL}}] \label{exactseq}
Let $\Fc$ be a saturated fusion system over a discrete $p$-toral group $S$. Let $\Qc$ and $\Rc$ be $\Fc$-invariant intervals in $\Fc^c$ such that:
\begin{enumerate}[(i)]
\item $\Qc\cap\Rc=\emptyset$;
\item $\Qc\cup\Rc$ is an interval in $\Fc^c$;
\item if $Q\in\Qc$, and $R\in\Rc$, then $Q\not\leq R$.
\end{enumerate}
Then there is a short exact sequence of functor $\xymatrix{0\ar[r]& \Zc^\Rc_{\Fc^c}\ar[r] & \Zc^{\Qc\cup\Rc}_{\Fc^c}\ar[r] & \Zc^\Qc_{\Fc^c}\ar[r]& 0}$ and a long exact sequence
\[\xymatrix{0\ar[r] & L^0(\Fc;\Rc)\ar[r] & L^0(\Fc;\Qc\cup\Rc)\ar[r] & L^0(\Fc;\Qc)\ar[r] &\cdots\\
            \cdots\ar[r] & L^{k}(\Fc;\Rc)\ar[r] & L^{k}(\Fc;\Qc\cup\Rc)\ar[r] & L^{k}(\Fc;\Qc)\ar[r] &\cdots.}\]
In particular, the following statements hold.
\begin{enumerate}[(a)]
\item If $L^k(\Fc;\Rc)=0=L^k(\Fc;\Qc)$ for some $k\geq 0$, then $L^k(\Fc;\Qc\cup\Rc)=0$.
\item Assume $\Fc=\Fc_S(\Gamma)$ with $(\Gamma,S,Y)$ is a general setup, and $\Qc\cup\Rc=\Fc_{\geq Y}$. Then for any $k\geq 2$,
\[L^{k-1}(\Fc;\Qc)\cong L^k(\Fc;\Rc),\]
and there is a short exact sequence
\[\xymatrix{0\ar[r] & C_{Z(Y)}(\Gamma)\ar[r] & C_{Z(Y)}(\Gamma^*)\ar[r] & L^1(\Fc;\Rc)\ar[r] & 0,}
\]
where $\Gamma^*=\lbrace g\in\Gamma\mid P^g\in\Qc \text{ for some }P\in\Qc\rbrace$.
\end{enumerate}             
\end{lem}

\begin{lem}[{\cite[Lemma 1.20]{LL}}]\label{2.7}
Let $(\Gamma,S,Y)$ be a general setup, $\Gamma_1$ a normal subgroup of $\Gamma$ containing $Y$, and $S_1=S\cap\Gamma_1$. Set $\Fc=\Fc_S(\Gamma)$ and $\Fc_1=\Fc_{S_1}(\Gamma_1)$. Let $\Qc\subseteq\Sc_p(S)_{\geq Y}$ be an $\Fc$-invariant interval such that $S\in\Qc$ and such that $\Gamma_1\cap Q\in\Qc$ whenever $Q\in\Qc$. Set $\Qc_1=\{Q\in\Qc\mid Q\leq \Gamma_1\}$. Then restriction induces an injection
\[
\xymatrix{L^1(\Fc;\Qc)\ar[r] & L^1(\Fc_1;\Qc_1)}.
\]

\end{lem}

\subsection{Offenders}

\begin{defi}
For $S$ a discrete $p$-toral group set,
\begin{align*}
d(S)&=\text{sup}\{ \ord(A)\, ;\, A\text{ is an abelian subgroup of }S\},\\
\A(S)&=\{A\leq S \mid A\text{ is abelian and }\ord(A)=d(S)\}.
\end{align*}
We define the Thompson subgroup of $S$ by $J(S)= \langle\A(S)\rangle$. 
\end{defi}

\begin{defi}
Let $G$ be a finite group which acts faithfully on an abelian discrete $p$-toral group $D$.
An \emph{offender (in $G$ on $D$)} is an abelian $p$-subgroup $A\leq G$ such that 
\begin{enumerate}[(i)]
\item $D_0\leq C_D(A)$; and
\item $|A|\geq |D/C_D(A)|$.
\end{enumerate}
\end{defi}

\begin{defi}
Let $G$ be a finite group which acts faithfully on an abelian discrete $p$-toral group $D$.
A \emph{best offender} in $G$ on $D$ is an abelian $p$-subgroup $A\leq G$ such that,
\begin{enumerate}[(i)]
\item $D_0\leq C_D(A)$ and
\item $|A|\cdot|C_D(A)/D_0|\geq |B|\cdot|C_D(B)/D_0|$ for each $B\leq A$. 
\end{enumerate} 
Let $\A_D(G)$ be the set of all nontrivial best offenders in $G$ on $D$.
\end{defi}

By taking $B=1$ in (ii), one sees that a best offender is in particular an offender.
Conversely, every nontrivial offender contains a nontrivial best offender: take $B\leq A$ such that $|B||C_D(B)/D_0|$ is maximal among all nontrivial subgroups of $A$.

An important result about offenders is Timmesfeld's Replacement Theorem. Levi and Libman gave a generalization of this to actions on discrete $p$-toral groups.

\begin{defi}
Let $G$ be a finite group which acts on an abelian group $V$. We say that $G$ acts \emph{quadratically} if $[V,G,G]:=[[V,G],G]=1$.
If $G$ acts faithfully on $V$, a best offender of $G$ on $V$ is a \emph{quadratic best offender} if it acts quadratically on $V$.
\end{defi}

\begin{thm}[{\cite[Theorem 2.11]{LL}}]\label{timmesfeld replacement}
Let $V$ be a nontrivial abelian discrete $p$-toral group and let $A$ be a nontrivial finite abelian $p$-group.
Suppose $A$ acts faithfully on $V$ and that $A$ is a best offender on $V$.
Then there exists a nontrivial subgroup $B\leq A$ such that $B$ is a quadratic best offender on $V$.
\end{thm}

The following lemma will allow us to navigate between action on an abelian discrete $p$-toral group and an action on a finite abelian $p$-group.

\begin{lem}[{\cite[Lemma 2.7]{LL}}]\label{action on discrete=action on finite}
Let $A$ be a finite group acting on an abelian discrete $p$-toral group $V$. Then there is some $N>0$ such that the following statements hold for all $n\geq N$:
\begin{enumerate}[(a)]
\item $V=V_0+\Omega_n(V)$;
\item $C_A(V)=C_A(\Omega_n(V))$.
\end{enumerate}
If in addition $A$ acts faithfully on $V$ and trivially on $V_0$, then
\begin{enumerate}
\item[(c)] for any $B\leq A$, $|C_{\Omega_n(V)}(B)|=|C_V(B)/V_0|\cdot|\Omega_n(V_0)|$ and hence,
\item[(d)] any $B\leq A$ is a (quadratic) best offender on $V$ if and only if $B$ is a (quadratic) best offender on $\Omega_n(V)$.
\end{enumerate}
\end{lem}

We finish with some notations for subgroups which are constructed in the same way as the Thompson subgroup.
Here $\A$ will be, in general, a set of offenders in $G$ on $D$.

\begin{defi}
Let $(\Gamma,S,Y)$ be a general setup, set $D=Z(Y)$ and $G=\Gamma/C_\Gamma(D)$.
Let $\A$ be a $G$-invariant set of abelian subgroups of $G$.
For any subgroup $H\leq G$, let $\A\cap H=\{A\in\A \mid A\leq H\}$ and $J_\A(H)=\langle\A\cap H\rangle$.
For any subgroup $\Gamma_1\leq \Gamma$, set $J_\A(\Gamma_1,D)$ (or $J_\A(\Gamma_1)$ when $D$ is understood) equal to the preimage in $\Gamma_1$ of $J_\A(\Gamma_1 C_\Gamma(D)/C_\Gamma(D))$.
\end{defi}

\begin{rem}\label{inequalitiesforJ}
If $H'$ is a subgroup of $H$ then $J_\A(H')\subseteq J_\A(H)$ and if $J_\A(H)\leq H'\leq H$ then $J_\A(H)=J_\A(H')$. In particular, $J_\A( J_\A(H))=J_\A(H)$. 
\end{rem}

\section{The odd case}
\subsection{The Norm Argument}
For a finite group $G$ which acts on an abelian discrete $p$-toral group $V$ (written multiplicatively), and a subgroup $H$ of $G$, define the \emph{norm map} $\Nc_H^G:C_V(H)\rightarrow C_V(G)$ by 
\[\Nc_H^G(v)=\prod_{g\in[G/H]}v^g\]
for each $v\in V$ and where $[G/H]$ is a set of representatives of the right cosets of $H$ in $G$. As the domain  is $C_V(H)$, the definition of $\Nc_H^G$ is independent of the choice of representatives. We say that $\Nc_H^G=1$ on $V$ if $\Nc_H^G$ is the constant map equal to the identity element of $V$. Since $\Nc_H^G=\Nc_K^G\Nc_H^K$ whenever $H\leq K\leq G$, one sees that $\Nc_H^G=1$ on $V$ whenever $\Nc_H^K$ or $\Nc_K^G$ is 1 on $V$.

\begin{thm}[{\cite[Theorem 3.2]{GL}}]\label{FiniteNormArg1}
Let $G$ be a finite group, $S\in\Syl_p(G)$, and $P$ be a $p$-group on which $G$ acts.
Let $\A$ be a nonempty set of subgroups of $S$, and set $J=\langle\A\rangle$.
Assume that $J$ is $\Fc_S(G)$-weakly closed and that
\begin{align}\label{NormCondition}
\text{if $A\in\A$, $g\in G$, $A\not\leq S^g$, and $V$ is a composition factor of $P$ under $G$, then $\Nc^A_{A\cap S^g}=1$ on $V$.}
\end{align} 
Then $C_P(G)=C_P(N_G(J))$.
\end{thm}

\begin{proof}
This is stated in Glauberman and Lynd \cite[Theorem 3.2]{GL}. A proof can be found in \cite[Theorem A1.4]{Gl71}.
\end{proof}

\begin{cor}\label{NormArg1}
Let $G$ be a finite group, $S\in\Syl_p(G)$, and let $D$ be an abelian discrete $p$-toral group on which $G$ acts.
Let $\A$ be a nonempty set of subgroups of $S$, and set $J=\langle\A\rangle$.
Assume that $J$ is $\Fc_S(G)$-weakly closed and that there exists $N>0$ such that, for every $n\geq N$, condition \eqref{NormCondition} is satisfied for $P=\Omega_n(D)$.
Then $C_D(G)=C_D(N_G(J))$.  
\end{cor}

\begin{proof}
By Lemma \ref{action with discreteptoral} and Theorem \ref{FiniteNormArg1}, we have
\[C_D(G)=\bigcup_{n\geq 1} C_{\Omega_n(D)}(G)=\bigcup_{n\geq 1} C_{\Omega_n(D)}(N_G(J))=C_D(N_G(J)).\]
\end{proof}

\subsection{Application of the Norm Argument}

\begin{lem}[{cf. \cite[Lemma 3.8]{GL}}]\label{barnorm}
Let $(\Gamma,S,Y)$ be a general setup for the prime $p$.
Let $D=Z(Y)$ and for every $X\subseteq \Gamma$ denote by $\overline{X}$ its image in $G=\Gamma/C_\Gamma(D)$.
If $Q\leq S$ is a subgroup containing $C_S(D)$, then $\overline{N_\Gamma(Q)}=N_G(\overline{Q})$.
\end{lem}

\begin{proof} The proof is exactly the same as that of \cite[Lemma 3.8]{GL}. The key is  the Frattini Argument, which is also true for virtually discrete $p$-toral group (Lemma \ref{FrattiniArg} part (b)).

Let $N$ be the preimage of $N_G(\overline{Q})$ in $\Gamma$. We have clearly $C_\Gamma(D)N_\Gamma(Q)\leq N$. For the other inclusion, notice that $QC_\Gamma(D)$ is normal in $N$ and that $QC_\Gamma(D)\cap S=QC_S(D)=Q$ ($Q$ contains $C_S(D)$) is a Sylow $p$-subgroup of $QC_\Gamma(D)$ ($D$ is normal in $\Gamma$ and thus fully centralized in $\Fc_S(\Gamma)$). Hence, by the Frattini Argument, $N=(QC_\Gamma(D))N_N(Q)\leq QC_\Gamma(D)N_\Gamma(Q)=C_\Gamma(D)N_\Gamma(Q)$.
\end{proof}

When $(\Gamma,S,Y)$ is a reduced setup, $C_S(D)=Y$. Thus, by Lemma \ref{barnorm}, for every subgroup $P\in\Sc_p(S)_{\geq Y}$, $N_G(\overline{P})=\overline{N_\Gamma(P)}$.

\begin{lem}[{\cite[Lemma 3.1]{GL}}]\label{pactsonp'el}
Let $A$ be a $p$-group acting on an elementary abelian $p$-group $V$.
\begin{enumerate}[(a)]
\item If $p$ is odd and $A$ acts quadratically on $V$, then $\Nc_{A_0}^A=1$ on $V$ for every proper subgroup $A_0$ of $A$.
\item If $p=2$, then $\Nc_{A_0}^A=1$ on $V$ for every subgroup $A_0$ of $A$ satisfying either of the following conditions:
\begin{enumerate}[(i)]
\item $|A:A_0|\geq 2$ and $C_V(A_0)=C_V(A)$, or
\item $|A:A_0|\geq 4$ and $A$ acts quadratically on $V$.
\end{enumerate}
\end{enumerate}
\end{lem}

The following is a generalization of \cite[Proposition 3.3]{LL} for odd primes.
\begin{prop}[{cf. \cite[Proposition 3.9]{GL}}]\label{oddcase k=1}
Let $(\Gamma,S,Y)$ be a reduced setup, set $D=Z(Y)$, $G=\Gamma/C_\Gamma(D)$ and $\Fc=\Fc_S(\Gamma)$. Let $\A$ be a $G$-invariant collection of $p$-subgroups of $G$, each of which acts non-trivially and quadratically on $D$. Let $\Rc\subseteq \Sc_p(S)_{\geq Y}$ be an $\Fc$-invariant interval of subgroups of $S$ such that $Y\in \Rc$ and $J_\A(S)\not\in \Rc$.
If $p$ is odd, then 
\[L^1(\Fc,\Rc)=0.\]
\end{prop}

\begin{proof}
Set $\Qc$ the set of subgroups of $S$ which contain $Y$ and are not in $\Rc$.
Let $\Gamma^*\subseteq \Gamma$ be the set of those $g\in\Gamma$ for which there is $Q\in\Qc$ with $Q^g\in \Qc$. Then $\Qc$ and $\Rc$ are $\Fc$-invariant intervals which satisfy the hypothesis of Lemma \ref{exactseq}. Hence it is enough to show  that $C_D(\Gamma)=C_D(\Gamma^*)$ by part (b) of Lemma \ref{exactseq}.

Let $\overline{S}$ be the image of $S$ in $G=\Gamma/C_\Gamma(D)$. As $\A$ is $G$-invariant, $J_\A(\overline{S})$ is $\Fc_{\overline{S}}(G)$-weakly closed  and, as each $A\in\A$ acts quadratically on $D$, by Lemma \ref{pactsonp'el} the hypotheses of Corollary \ref{NormArg1} are satisfied. Then, 
 \[C_D(\Gamma)=C_D(G)=C_D(N_G(J_\A(\overline{S})))\]
As $C_S(D)\leq J_\A(S,D)=J_\A(S)$ by definition, by Lemma \ref{barnorm} $N_G(J_\A(\overline{S}))=N_\Gamma(J_\A(S))$ and then $C_D(\Gamma)= C_D(N_\Gamma(J_\A(S))$.
Since $J_\A(S,D)\in\Qc$ by assumption, we have $N_\Gamma(J_\A(S))\subseteq \Gamma^*$ by the definition of $\Gamma^*$. Thus, 
\[C_D(\Gamma)=C_D(N_\Gamma(J_\A(S)))\geq C_D(\Gamma^*)\geq C_D(\Gamma)\]
(the last inequality holds because $\Gamma^*\subseteq\Gamma$) and this completes the proof.  
\end{proof}

As in the finite case, we can obtain some simplifications in the proof of \cite[Proposition 3.5]{LL}, when $p$ is odd.

\begin{prop}\label{oddcase}
Let $(\Gamma,S,Y)$ be a general setup for the prime $p$.
Set $\Fc=\Fc_S(\Gamma)$, $D=Z(Y)$ and $G=\Gamma/C_\Gamma(D)$. Let $\Rc\subseteq \Sc_p(S)_{\geq Y}$ be an $\Fc$-invariant interval such that for all $Q\in\Sc_p(S)_{\geq Y}$, $Q\in \Rc$ if and only if $J_{\A_D(G)}(Q)\in\Rc$. If $p$ is odd, then $L^k(\Fc;\Rc)=0$ for all $k\geq 1$.
\end{prop}

\begin{proof}
Let $(\Gamma,S,Y,\Rc,k)$ be a counter-example for which the four-tuple $(k,\ord(\Gamma),|\Gamma/Y|,|\Rc|)$ is minimal with respect to the lexicographic order. Steps 1, 2 and 3 in the proof of \cite[Proposition 3.5]{LL} show that $\Rc=\{R\in \Sc_p(S)_{\geq Y}\mid J_{\A_D(G)}(R)=Y\}$, $k=1$ (since $p$ is odd) and $(\Gamma,S,Y)$ is a reduced setup.

Let then $\A$ be the set of nontrivial quadratic best offenders in $G$ on $D$. By Timmesfeld's Replacement Theorem (Theorem \ref{timmesfeld replacement}) each nontrivial best offender contains a quadratic best offender.
If $J_{\A_D(G)}(R)=Y$ (respectively, $J_{\A}(R)=Y$)  then $\overline{R}$, the image of $R$ in $G$, does not contain a non-trivial best offender (respectively, non-trivial quadratic best offender) of $G$ on $D$. Hence, $J_\A(R)=Y$ if and only if $J_{A_D(G)}(R)=Y$.

Now, if $S\in\Rc$ then, since $\Rc$ is an interval, $\Rc=\Sc_p(S)_{\geq Y}$ and $L^k(\Fc;\Rc)=0$ for all $k>1$ by Lemma \ref{higherlimitswithS} (b). Hence $S\not\in\Rc$ and so $J_\A(S)\not\in\Rc$. Then Proposition \ref{oddcase k=1} show that $(\Gamma,S,Y,\Rc,1)$ is not a counterexample. 
\end{proof}

\section{The even case}

From now on, $p=2$.

\subsection{Norm Arguments for \texorpdfstring{$p=2$}{Lg}}\label{normargp=2 section}

Here, we need a modified version of the Norm Argument (Corollary \ref{NormArg1}).
\begin{thm}[{\cite[Theorem 4.1]{GL}}]\label{normargumentp=2finite}
Suppose $G$ is a finite group, $S\in\Syl_2(G)$, and let $D$ be an abelian $2$-group on which $G$ acts. Let $\A$ be a nonempty set of subgroups of $S$, and set $J=\langle \A\rangle$. Let $H$ be a subgroup of $G$ containing $N_G(J)$. Set $V=\Omega_1(D)$. Assume that $J$ is weakly closed in $S$ with respect to $G$, and
\begin{equation*}
\text{whenever }A\in\A,\, g\in G,\text{ and }A\not\leq H^g,\text{ then }\Nc_{A\cap H^g}^A=1\text{ on }V, 
\end{equation*}
or more generally,
\begin{equation*}
\text{whenever }g\in G\text{ and }J\not\leq H^g,\text{ then }\Nc_{J\cap H^g}^J=1\text{ on }V.
\end{equation*}
Then $C_D(H)=C_D(G)$.
\end{thm}

\begin{cor}\label{normargumentp=2}
Suppose $G$ is a finite group, $S\in\Syl_2(G)$, and let $D$ be an abelian $2$-toral group on which $G$ acts. Let $\A$ be a nonempty set of subgroups of $S$, and set $J=\langle \A\rangle$. Let $H$ be a subgroup of $G$ containing $N_G(J)$. Set $V=\Omega_1(D)$. Assume that $J$ is weakly closed in $S$ with respect to $G$, and
\begin{equation}\label{4.2}
\text{whenever }A\in\A,\, g\in G,\text{ and }A\not\leq H^g,\text{ then }\Nc_{A\cap H^g}^A=1\text{ on }V, 
\end{equation}
or more generally,
\begin{equation}\label{4.3}
\text{whenever }g\in G\text{ and }J\not\leq H^g,\text{ then }\Nc_{J\cap H^g}^J=1\text{ on }V.
\end{equation}
Then $C_D(H)=C_D(G)$.
\end{cor}

\begin{proof}
For every $n\geq 1$, $\Omega_1(\Omega_n(D))=\Omega_1(D)=V$. Hence, by Theorem \ref{normargumentp=2finite} for every $n\geq 1$, 
$C_{\Omega_n(D)}(H)=C_{\Omega_n(D)}(G)$ and thus, as $D=\cup_{n\geq 1}\Omega_n(D)$ by Lemma \ref{action with discreteptoral}(a), the corollary holds. 
\end{proof}

\begin{defi} Let $G$ be a finite group, $S\in\Syl_2(G)$ and let $D$ be an abelian $2$-toral group on which $G$ acts faithfully. A subgroup $A\in \A_D(G)$ is an \emph{over-offender} of $G$ on $D$ if $|A||C_D(A)/D_0|>|D/D_0|$. 
Set
\[\hat{\A}_D(G)=\{A\in\A_D(G)\mid |A||C_D(A)/D_0|>|D/D_0|\}.\]

For $\A$ a collection of subgroups of $G$, we denote by $\A^\circ$ the collection of members of $\A$ which are minimal under inclusion. We also write 
\begin{align*}
\A_D(G)_2 &=\{A\in\A_D(G)\mid|A|=2\};\text{ and}\\
\A_D(G)^{\circ}_{\geq 4}&=\{A\in\A_D(G)^\circ\mid|A|\geq 4\}.
\end{align*}
\end{defi}
Be careful that in general  $\hat{\A}_D(G)^\circ\not\subseteq \A_D(G)^\circ$.
Notice also that, by Theorem \ref{timmesfeld replacement}, every member of $\A_D(G)^\circ$ acts quadratically on $D$.

\begin{rem}\label{overoffender finite=discrete}
By Lemma \ref{action on discrete=action on finite}, we have $D=\Omega_n(D)D_0$ when $n$ is large enough and thus $\hat{\A}_D(G)=\hat{\A}_{\Omega_n(D)}(G)$ for $n$ large enough.
\end{rem}

\begin{rem}\label{4.5}
Assume $G$ acts faithfully on $D$. If $A\in\A_D(G)_2$, then $|D/C_D(A)|=2$. In particular, every member of $\hat{\A}_D(G)$ is of order at least $4$.
\end{rem}

\begin{lem}[{cf. \cite[Lemma 4.6]{GL}}]\label{4.6}
Let $G$ be a finite group, $S\in\Syl_2(G)$ and let $D$ be an abelian discrete $2$-toral group on which $G$ acts faithfully. Let $\A=\A_D(G)_{\geq 4}^\circ\cup\hat{\A}_D(G)^\circ$ and $V=\Omega_1(D)$. Assume that $\A$ is nonempty and $H$ is a subgroup of $G$ containing $N_G(J_\A(S))$. Then $\A\cap S$ satisfies \eqref{4.2}.
\end{lem}

\begin{proof}
Fix $A\in \A\cap S$ and $g\in G$ such that $A\neq H^g$ and let $A_0$ be a subgroup of $A$ of index 2 that contains $A\cap H^g$. Suppose $\Nc_{A\cap H^g}^A$ is not 1 on $V$. Since $\Nc_{\A\cap H^g}^A=\Nc_{A_0}^A\Nc_{A\cap H^g}^{A_0}$, $\Nc_{A_0}^A$ is not 1 on $V$ and we have $C_V(A)<C_V(A_0)$ by Lemma \ref{pactsonp'el}(b). Thus, as $A$ acts trivially on $D_0$, we have $C_D(A)/D_0< C_D(A_0)/D_0$ which is equivalent to $2|C_D(A)/D_0|\leq |C_D(A_0)/D_0|$. Therefore, we obtain
\begin{equation*}\label{4.7}
|A_0||C_D(A)/D_0|\geq \frac{1}{2}|A|\cdot 2|C_D(A)/D_0|=|A||C_D(A)/D_0|.
\end{equation*}
This inequality implies that either $A_0=1$ or $A_0$ is a best offender on $D$ in that later case $A_0$ is also an over-offender. The first case contradicts the fact that $A\in \A_D(G)_{\geq 4}$ and the second case contradicts the fact that $A\in \hat{\A}_D(G)^\circ$.
\end{proof}

\begin{lem}[{cf. \cite[Lemmas 4.8 and 4.12]{GL}}]\label{4.8}
Let $G$ be a finite group, $S\in\Syl_2(G)$ and let $D$ be an abelian discrete $2$-toral group on which $G$ acts faithfully. Let $\A=\A_D(G)_2$. Assume that $\hat{\A}_D(G)=\emptyset$ and that $A,B\in\A$.
\begin{enumerate}[(a)]
\item If $[A,B]=1$ and $A\neq B$, then $C_D(A)\neq C_D(B)$ and $AB$ is quadratic on $D$.
\item If $\langle A,B\rangle$ is a 2-group, then $[A,B]=1$.
\item $J_\A(S)$ is elementary abelian.
\item If $[A,B]\neq 1$, then $L:=\langle A,B\rangle\cong S_3$, $[D,L]$ is elementary abelian of order $4$, and $D=[D,L]\times C_D(L)$.
\end{enumerate}
\end{lem}

\begin{proof}Since $A$ and $B$ lies in $\A_D(G)_2$,
\begin{equation}
[D,A,A]=1=[D,B,B]\text{ and }|D/C_D(A)|=2.
\end{equation}
The statements (a), (b) and (c) follow from Lemma \ref{action on discrete=action on finite} and from \cite[Lemma 4.8]{GL} (applied with $D=\Omega_n(D)$ with $n$ large enough). 

To prove (d), set $L=\langle A,B\rangle$ and, by Lemma \ref{action on discrete=action on finite} (d),  fix $n$ large enough such that $\A_D(G)=\A_{\Omega_n(D)}(G)$ and $D=D_0+\Omega_n(D)$. Then, by \cite[Lemma 4.8(c)]{GL}, $L\cong S_3$. Moreover, by definition of offenders, $A$ and $B$ act trivially on $D_0$. Thus, by \cite[Lemma 4.8(c)]{GL},  $[D,L]=[\Omega_n(D),L]$ is elementary abelian of order 4, and $D=[D,L]\times C_D(L)$.  
\end{proof}

\begin{defi}\label{solitary}
Let $G$ be a finite group, $S\in\Syl_2(G)$ and let $D$ be an abelian discrete $2$-toral group on which $G$ acts faithfully. Let $\A=\A_D(G)_2$, $J=J_\A(S)$ and $T\in\A\cap S$. We say that $T$ is \emph{solitary in $G$ relative to $S$} if there is a group $L$ in $G$ containing $T$ such that:
\begin{enumerate}[(i)]
\item $L\cong S_3$;
\item $J=T\times C_J(L)$ and$C_J(L)=\langle (\A\cap S)\smallsetminus \{T\}\rangle$; and
\item $D=[D,L]\times C_D(L)$ and $[D,L,C_J(L)]=1$.
\end{enumerate}
Denote by $\T_D(G)$ the collection of subgroups of $G$ that are solitary relative to a Sylow $2$-subgroup of $G$. Set $\Bc_D(G)=\A\smallsetminus\T_D(G)=\{A\in\A_D(G)\mid|A|=2 \text{ and }A\not\in\T_D(G)\}$. 
\end{defi}

\begin{defi}\label{semisolitary}
Let $G$ be a finite group, $S\in\Syl_2(G)$ and let $D$ be an abelian discrete $2$-toral group on which $G$ acts faithfully. Let $S_0\leq S$ and set $\A=\A_D(G)_2$.
We say that a subgroup $T\leq S_0$ is \emph{semisolitary relative to $S_0$} if there are subgroups $W$ and $X$ of $D$ that are normalized by $\langle \A\cap S_0\rangle$, such that
\begin{enumerate}[(i)]
\item $\langle \A\cap S_0\rangle=T\times \langle (\A\cap S_0)\smallsetminus\{T\}\rangle$; and
\item $W$ is elementary abelian of order $4$, $D=W\times X$, $T$ centralizes $X$, and $\langle (\A\cap S_0)\smallsetminus\{T\}\rangle$ centralizes $W$.
\end{enumerate}
\end{defi}

\begin{rem}\label{4.13}
If $\hat{\A}_D(G)=\emptyset$, then given a solitary offender $T$ of $G$ relative to $S$, and given $L$ as in Definition \ref{solitary}, we can take $W=[D,L]$ and $X=C_D(L)$, and see from Lemma \ref{4.8}(d) that $T$ is semisolitary relative to some Sylow 2-subgroup of $G$. 
\end{rem}

\begin{lem}[{cf. \cite[Lemma 4.15]{GL}}]\label{4.14}
Let $G$ be a finite group, $S\in\Syl_2(G)$ and let $D$ be an abelian discrete $2$-toral group on which $G$ acts faithfully.
Assume that $\hat{\A}_D(G)=\emptyset$. Then $\T_D(G)\cap S$ is the set of subgroups that are solitary relatively to $S$. 
\end{lem}

\begin{proof}
By Definition \ref{solitary}, if $T$ is solitary relatively to $S$ then $T\in\T_D(G)\cap S$. For the converse, set $\A=\A_D(G)_{2}$ and assume $T\in\T_D(G)\cap S$. Therefore, $T$ is solitary relatively to a Sylow $2$-subgroup $S_1$ of $G$. Choose  $g\in G$ such that $S_1^g=S$ and set $J_1=\langle \A\cap S_1\rangle$ and $J=\langle \A\cap S\rangle$. Then $J_1$ is elementary abelian by Lemma \ref{4.8}, and $T^g\leq J_1^g=J$. Fix a subgroup $L_1\cong S_3$ containing $T$ and which satisfies the axioms of Definition \ref{solitary} with $S_1$ and $J_1$ in the roles of $S$ and $J$, respectively. Since $J$ is abelian and weakly closed in $S$ with respect to $\Fc_S(G)$, there is $h\in N_G(J)$ such that $T^{gh}=T$ by \cite[Lemma 8.1]{GL}. Setting $L=L_1^{gh}$, one establishes the validity of the axioms of Definition \ref{solitary} for $L$, $S$ and $J$ from their validity for $L_1$, $S_1$ and $J_1$.
\end{proof}

\begin{lem}\label{lem solitary}
Let $G$ be a finite group, $S\in\Syl_2(G)$, and let $D$ be an abelian discrete $2$-toral group on which $G$ acts faithfully. Then there exists $N\geq 1$ such that,  for all $n\geq N$, 
$\T_D(G)=\T_{\Omega_n(D)}(G)$ and $ \Bc_D(G)=\Bc_{\Omega_n(D)}(G)$.
\end{lem}

\begin{proof}
By Lemma \ref{action on discrete=action on finite} (d), there exists $N$ such that, for every $n\geq N$, we have $D=\Omega_n(D)D_0$ and $\A_D(G)=\A_{\Omega_n(D)}(G)$. The two equalities are equivalent so we just prove that, for all $n\geq N$, $\T_D(G)=\T_{\Omega_n(D)}(G)$.
Let $n\geq N$ and $A\in \A_D(G)_2=\A_{\Omega_n(D)}(G)_2$. 

Assume first that $A\in \T_D(G)$ and set $L$ and $J$ as in Definition \ref{solitary}. We will prove that $A\in \T_{\Omega_n(D)}(G)$ using the same $L$ and $J$. The two first conditions (i) and (ii) of Definition \ref{solitary} are satisfied and only the third condition needs to be handled. 
Since $L$ is generated by $A$ and one of its conjugates, $L$ is generated by offenders and thus $D_0\leq C_D(L)$. In particular, since $D=\Omega_n(D)D_0$ we have $[D,L]=[\Omega_n(D),L]$. 
Therefore,
\[
\Omega_n(D)=([D,L]\times C_D(L))\cap \Omega_n(D)=[\Omega_n(D),L]\times C_{\Omega_n(D)}(L)
\]
and
\[[\Omega_n(D),L,C_J(L)]=[[\Omega_n(D),L],C_J(L)]=[[D,L],C_J(L)]=1.\]
Thus $A\in \T_{\Omega_n(D)}(G)$.

Assume now that $A\in \T_{\Omega_n(D)}(G)$ and set $L$ and $J$ as in Definition \ref{solitary}. Here again, we will prove that $A\in \T_D(G)$ using the same $L$ and $J$. The two first conditions (i) and (ii) of Definition \ref{solitary} are independent of $D$ (or $\Omega_n(D)$) thus only the third condition needs to be handled. As $L$ is generated by offenders in $\A_{\Omega_n(D)}(G)_2=\A_D(G)_2$, $D_0\leq C_D(L)$ and thus $[D,L]=[\Omega_n(D),L]$. Thus
\[[D,L,C_J(L)]=[[D,L],C_J(L)]=[[\Omega_n(D),L],C_J(L)]=1.\]
Moreover, since $D_0\leq C_D(L)$, $\Omega_n(D)\cap D_0\leq C_{\Omega_n(D)}(L)$ and then
\[
D=\Omega_n(D)D_0=([\Omega_n(D),L]\times C_{\Omega_n(D)}(L))D_0=[D,L]\times C_{\Omega_n(D)}(L)D_0=[D,L]\times C_D(L).
\]
Therefore $A\in \T_D(G)$.
\end{proof}

\begin{lem}[{cf. \cite[Lemma 4.16]{GL}}]\label{4.15}
Let $G$ be a finite group, $S\in\Syl_2(G)$ and let $D$ be an abelian discrete $2$-toral group on which $G$ acts faithfully.
Set $\A=\A_D(G)_2$, $\T=\T_D(G)$, $\Bc=\Bc_D(G)=\A\smallsetminus \T$.
Assume $\hat{\A}_D(G)=\emptyset$ and $\Bc\neq\emptyset$.
Then for each subgroup  $J^*\leq J_\Bc(S)$ that is weakly closed with respect to $\Fc_S(G)$, $\A\cap S$ satisfies \eqref{4.3} with $H=N_G(J^*)$.
\end{lem}

\begin{proof}
By Lemma \ref{action on discrete=action on finite} and Remark \ref{overoffender finite=discrete}, $\A_D(G)_2=\A_{\Omega_D(G)}(G)_2$ and $\hat{\A}_{\Omega_n(D)}(G)=\hat{\A}_D(G)=\emptyset$ when $n$ is large enough. Moreover by Lemma   \ref{lem solitary}, with $n$ large enough, $\Bc_{\Omega_n(D)}(G)=\Bc\neq \emptyset$. Let us fix $n$ large enough such that $\A_D(G)_2=\A_{\Omega_n(D)}(G)_2$, $\hat{\A}_{\Omega_n(D)}(G)=\hat{\A}_D(G)=\emptyset$ and $\Bc_{\Omega_n(D)}(G)=\Bc\neq \emptyset$. 
Then, by \cite[Lemma 4.16]{GL}, for every subgroup $J^*\leq J_{\Bc_{\Omega_n(D)}(G)}(S)=J_\Bc(S)$ such that $J^*$ is weakly closed with respect to $G$, $\A_{\Omega_n(D)}(G)\cap S=\A\cap S$ satisfies \eqref{4.3} (recall that in Theorem \ref{normargumentp=2} $V=\Omega_1(D)=\Omega_1(\Omega_n(D))$).
Therefore, the Lemma follows.
\end{proof}

\begin{lem}[{cf. \cite[Lemma 4.29]{GL}}]\label{4.29}
Let $G$ be a finite group, $S\in\Syl_2(G)$ and let $D$ be an abelian discrete $2$-toral group on which $G$ acts faithfully. Let $P\leq S$ and let $\T\subseteq\A_D(G)_2\cap P$ be the collection of semisolitary subgroups relative to $P$. Assume $A\in\A_D(G)^\circ\cup\hat{\A}_D(G)^\circ$ normalizes $P$. Then $A$ normalizes every element of $\T$.
\end{lem}

\begin{proof}
By Lemma \ref{action on discrete=action on finite}, there is $N>0$ such that for all $n\geq N$, $D=\Omega_n(D)D_0$ and $\A_D(G)=\A_{\Omega_n(D)}(G)$. We prove that if $n\geq \max(2,N)$ and $T\in \T$ then $T$ is semi-solitary with respect to $P$ over $\Omega_n(D)$ and then, the Lemma will follow from \cite[Lemma 4.29]{GL}. 
Fix $n\geq \max(2,N)$, assume that $T\in \T$ and set $W$ and $X$ as in Definition \ref{semisolitary}. We want to prove that $T$ is semi-solitary with respect to $P$ over $\Omega_n(D)$ using $W$ and $X_n=X\cap\Omega_n(D)$. As $\A_D(G)=\A_{\Omega_n(D)}(G)$, (i) is satisfied and it remains to prove (ii). To simplify, set $\A:=\A_D(G)_2=\A_{\Omega_n(D)}(G)_2$. Since $\Omega_n(D)$ is a characteristic subgroup of $D$, $W$ and $X_n$ are normalized by $\langle S_0\cap \A\rangle$ and $T$ centralizes $X_n$. Moreover, since $W$ is elementary abelian, $W\leq \Omega_n(D)$. Thus $\Omega_n(D)= (W\times X)\cap \Omega_n(D)=W\times X_n$ and (ii) is satisfied.
\end{proof}

\subsection{Conjugacy families and conjugacy functors}\label{section conj functand fam}

We will need to use the notion of conjugacy functor, or more precisely, we need an extended notion of conjugacy functor. This notion is introduced in \cite[Section 5]{Gl71} but one may also find some details in \cite[Appendix 2]{GL}. Here, we define the notion of conjugacy functor with respect to a fusion system instead of a group. In fact, almost everything about conjugacy functors in \cite[Section 5]{Gl71} can be phrased in terms of saturated fusion systems.

\begin{defi}\label{conjugacy functor}
Let $\Fc$ be a fusion system over a discrete $p$-toral group $S$. An \emph{$\Fc$-conjugacy functor} is a map $W:\Sc_p(S)\rightarrow \Sc_p(S)$ such that for all $P\leq S$,
\begin{enumerate}[(i)]
\item $W(P)\leq P$;
\item $W(P)\neq 1$ whenever $P\neq 1$; and
\item $W(P)^\varphi=W(P^\varphi)$ whenever $\varphi\in\Hom_\Fc(P,S)$.
\end{enumerate} 
\end{defi}

\begin{lem}[{cf. \cite[Lemma 8.5]{GL}}]\label{conj functor lem}
Let $\Fc$ be a fusion system over a discrete $p$-toral group $S$, and $W$ an $\Fc$-conjugacy functor. Then for all $P\leq S$,
\begin{enumerate}[(a)]
\item $N_S(P)\leq N_S(W(P))$;
\item $W(P)=W(N_S(W(P)))$ if and only if $W(P)=W(S)$; and
\item $P=N_S(W(P))$ if and only if $P=S$.
\end{enumerate} 
\end{lem}

\begin{proof} 
Part (a) follows from Definition \ref{conjugacy functor} (iii). 

For part (b), let $P\leq S$ and $T=N_S(W(P))$.
If $W(T)=W(P)$ then, by (a), $N_S(T)\subseteq N_S(W(T))=N_S(W(P))=T$. Thus $T=S$ by Lemma \ref{FrattiniArg}(a,) and $W(P)=W(T)=W(S)$. The converse is true, because $N_S(W(S))=S$ by (a).

Finally, we prove (c). If $P=T$, then again by (a), $N_S(P)\leq N_S(W(P))=T=P$, and then $P=S$ by Lemma \ref{FrattiniArg} (a). Finally, as $S=N_S(W(S))$, the converse holds.
\end{proof}

\begin{defi}
Let $\Fc$ be a fusion system over a discrete $p$-toral group $S$.
A \emph{conjugacy family} for $\Fc$ is a collection $\Cc$ of subgroups of $S$ such that every morphism in $\Fc$ is a composition of restrictions of $\Fc$-automorphisms of members of $\Cc$. 
\end{defi}

For example, by Alperin's Fusion Theorem (\cite[Theorem 3.6]{BLO3}), if $\Fc$ is saturated the collection of $\Fc$-centric-radical subgroups of $S$ is a conjugation family of $\Fc$.

\begin{defi}
Let $\Fc$ be fusion system over a discrete $p$-toral group $S$. Let $W$ be an $\Fc$-conjugacy functor. For all subgroup $P\leq S$ and for all $i\geq 1$, set recursively,
\[ W_1(P)=P\text{ and } W_i(P)=W(N_S(W_{i-1}(P))).\]
Then, $P$ is said to be \emph{well-placed with respect to $W$} if $W_i(P)$ is fully $\Fc$-normalized for all $i\geq 1$.
\end{defi}
By Lemma \ref{conj functor lem} (a), for $i\geq 2$, we have 
\begin{equation}\label{W_i inequality}
N_S(W_{i-1}(P))\leq N_S(N_S(W_{i-1}(P)))\leq N_S(W(N_S(W_{i-1}(P))))=N_S(W_i(P)).
\end{equation}
\begin{thm}\label{well placed subgroup thm}
Let $\Fc$ be a fusion system over a discrete $p$-toral group $S$, and $W$ an $\Fc$-conjugacy functor. If $\Fc$ is saturated, then
\begin{enumerate}[(a)]
\item every subgroup $P\leq S$ is $\Fc$-conjugate to a well-placed subgroup of $S$; and
\item the set of well-placed subgroups of $S$ forms a conjugation family for $\Fc$.
\end{enumerate}  
\end{thm}

\begin{proof}
(a).
Let $P\leq S$ and assume that $P$ is not conjugate to a well-placed subgroup of $S$.
By Lemma \ref{P^Fc finite S-conjugacy}, $P^\Fc$ contains finitely many $S$-conjugacy classes. Moreover every $S$-conjugate of a fully normalized subgroup of $S$ is also fully normalized. Hence, if $P,Q \leq S$ are $S$-conjugate, we have, by definition of $W_i$,
\[\min\{i\geq 1\mid W_i(P) \text{ is not fully normalized}\}=\min\{i\geq 1\mid W_i(Q) \text{ is not fully normalized}\}.\]
Let $Q\in P^\Fc$ be such that the $i_0=\min\{i\geq 1\mid W_i(Q) \text{ is not fully normalized}\}$ is maximum. Such a $Q$ exists and only depends on its $S$-conjugacy class from the preceding remarks. As $\Fc$ is saturated, by axiom (I) of Definition \ref{saturation}, there exists $W_{i_0}\in(W_{i_0}(Q))^\Fc$ which is fully normalized. By axiom (II) of Definition \ref{saturation}, $W_{i_0}$ is also fully automized and receptive. Hence by Lemma \ref{saturation lem}, there is $\varphi\in \Hom_\Fc(N_S(W_{i_0}(Q)),N_S(W_{i_0}))$ such that $\varphi(W_{i_0}(Q))=W_{i_0}$.
Recall that by \eqref{W_i inequality}, for every $i\geq 2$, $N_S(W_{i-1}(Q))\leq N_S(W_i(Q))$. In particular, $Q\leq N_S(W_{i_0}(Q))$. Consider $Q'=Q^\varphi\in P^\Fc$. For every $1\leq i\leq {i_0}$, by Definition \ref{conjugacy functor} of a conjugacy functor, we have $W_i(Q')=W_i(Q)^\varphi$ (in particular, $W_{i_0}(Q')=W_{i_0}$) and, $N_S(W_i(Q))^\varphi\leq N_S(W_{i_0})$ and normalizes $W_i(Q')$. Hence, for every $1\leq i\leq{i_0}$, $W_{i}(Q')$ is fully normalized, which contradicts the initial assumption on $P$.

(b). Assume that the set of well-placed subgroups is not a conjugacy family. Let $P$ be maximum among the subgroups of $S$ such that there is $Q\in P^\Fc$ and $\varphi\in\Iso_\Fc(P,Q)$ such that $\varphi$ is not a composite of restrictions of automorphisms of well-placed subgroups of $S$.
By Lemma \ref{conj functor lem}(a), we have, for $i\geq 2$, $W_i(S)=W(S)\unlhd S$. In particular, $S$ is well placed and we can assume that $P<S$ and thus $Q<S$. By (a) there exists $R\in P^\Fc$ well-placed and so fully normalized. 
By the saturation axioms (Definition \ref{saturation}), $R$ is also fully automized and receptive and by Lemma \ref{saturation lem}, there is $\varphi_P\in\Hom_\Fc(N_S(P),N_S(R))$ and $\varphi_Q\in\Hom_\Fc(N_S(Q),N_S(R))$ such that $P^{\varphi_P}=R$ and $Q^{\varphi_Q}=R$. 
we have then the following equality 
\[\varphi=\varphi_P\varphi_P^{-1}\varphi\varphi_Q\varphi_Q^{-1}.\] 
As $P<S$ and $Q<S$, by Lemma \ref{FrattiniArg} (a), $P<N_S(P)$ and $Q<N_S(Q)$. Then by maximality of $P$, $\varphi_P$ and $\varphi_Q$ are composites of restrictions of automorphisms of well-placed subgroups. Finally, $\varphi_P^{-1}\varphi\varphi_Q\in\Aut_\Fc(R)$ and $R$ is well-placed, thus $\varphi=\varphi_P\left(\varphi_P^{-1}\varphi\varphi_Q\right)\varphi_Q^{-1}$ is a composite of restrictions of automorphisms of well-placed subgroups of $S$. This contradicts the assumption on $P$ and $\varphi$.  
\end{proof}

\subsection{Reduction to the transvection case}
In this section we give in Theorem \ref{5.4} a proof of \cite[Proposition 3.5]{LL} for $p=2$ when some minimal offender under inclusion is not solitary. 

\begin{defi}
Let $\Fc$ be a saturated fusion system over a discrete $p$-toral group $S$ and let $\Qc\subseteq \Fc^c$ be an $\Fc$-invariant interval.
We say that a 1-cocycle for the functor $\Zc_\Fc^\Qc$ is \emph{inclusion-normalized} if it sends the class of any inclusion to the identity element.
\end{defi}

In the rest of the paper, all the cochains for the functor $\Zc_\Fc^\Qc$ we will consider are the identity on $\Fc$-centric subgroups outside of $\Qc$.
 
\begin{lem}[{cf. \cite[Lemma 5.1]{GL}}]\label{inclusion normalized cocycle lem}
Let $\Gamma$ be a virtually discrete $p$-toral group, $S\in\Syl_p(\Gamma)$ and $\Fc=\Fc_S(\Gamma)$.
Let $\Qc\subseteq\Fc^c$ be an $\Fc$-invariant interval such that $S\in\Qc$.
Let $\Gamma^*\subseteq \Gamma$ be the subset of those $g\in\Gamma$ for which there is $Q\in\Qc$ with $Q^g\in \Qc$.
Then each $1$-cocycle for $\Zc^\Qc_\Fc$ is cohomologous to an inclusion-normalized 1-cocycle.
Moreover, if $t$ is an inclusion-normalized 1-cocycle, then:
\begin{enumerate}[(a)]
\item $t([\varphi_1])=t([\varphi_2])$ for each commutative diagram
\[
\xymatrix{P_2\ar[r]^{\varphi_2} & Q_2\\
  P_1\ar[u]^{\iota_{P_1}^{Q_1}}\ar[r]^{\varphi_1} & Q_1\ar[u]_{\iota_{Q_1}^{Q_2}}}
\] in $\Fc$ among subgroups in $\Qc$.
\item The function $\tau:\Gamma^*\rightarrow\Gamma^*$ defined by the rule
\[g^\tau=t([c_g])g,\]
is a bijection and restricts to the identity map on $S$. Moreover:
\[(g_1g_2\dots g_n)^\tau=g_1^\tau g_2^\tau\dots g_n^\tau,\]
for each collection of elements $g_i\in\Gamma^*$ with the property that there is $Q\in\Qc$ such that $Q^{g_1\dots g_i}\in\Qc$ for all $1\leq i\leq n$.
\item $t=0$ if and only if $\tau$ is the identity on $\Gamma^*$.
\end{enumerate}
\end{lem}

\begin{proof}
Given a 1-cocycle $t$ for $\Zc_\Fc^\Qc$, define a 0-cochain $u$ by $u(P)=t([\iota_P^S])$ for each $P\in \Qc$.
Then, for any inclusion $\iota_P^Q$ in $\Fc$ with $P,Q\in\Qc$, 
\[du([\iota_P^Q])=u(Q)u(P)^{-1}=t([\iota_Q^S])t([\iota_P^S])^{-1},\]
so that
\[(tdu)([\iota_P^Q])=t([\iota_P^Q])t([\iota_Q^S])t([\iota_P^S])^{-1}=t([\iota_P^S])t([\iota_P^S])^{-1}=1,\]
by the 1-cocycle identity \ref{1-chainrule}. Hence $tdu$ is inclusion-normalized.

Assume now that $t$ is inclusion-normalized, and let $P_i$, $Q_i$, and $\varphi_i$ be as in (a). Recall that for $\varphi\in\Hom_\Fc(P,Q)$, we denote by $\varphi^{-1}$ the inverse of $\varphi\in\Hom_\Fc(P,P^\varphi)$. Since $t$ sends inclusions to the identity, the 1-cocycle identity \ref{1-chainrule} yields
\begin{align*}
t([\varphi_1\iota_{Q_1}^{Q_2}])&=t([\iota_{Q_2}^{Q_1}])^{\varphi_1^{-1}}t([\varphi_1])=t([\varphi_1]), \text{ and}\\
t([\iota_{P_1}^{P_2}\varphi_2])&=t([\varphi_2]t([\iota_{P_1}^{P_2}])=([\varphi_2]),
\end{align*}
and so (a) follows by commutativity of the given diagram.

Let $\tau$ be given as in (b). Since $g\in\Gamma^*$, the conjugation map $c_g:{}^gS\cap S\rightarrow S\cap S^g$ is a map between subgroups in $\Qc$. Part (a) shows that $t([c_g])$ agrees with the value of $t$ on the class of each restriction of $[c_g]$ provided that the source and target of such a restriction lies in $\Qc$.
This shows that $\tau$ is well defined.
Then $\tau$ is a bijection since its inverse is induced by $t^{-1}$ (which is inclusion-normalized) in the same way. Besides, for $s\in S$, $[c_s]=[\Id_S]$ is the identity in the orbit category. Thus, since $t$ is inclusion-normalized, $\tau$ is the identity map on $S$.

Let $g_1,g_2\in\Gamma^*$ and $Q\in\Qc$ with $Q^{g_1}\leq S$ and $Q^{g_1g_2}\leq S$. Then, by the 1-cocycle identity \ref{1-chainrule},
\[g_1^\tau g_2^\tau=t([c_{g_1}])g_1t([c_{g_2}])g_2=t([c_{g_1}])t([c_{g_2}])^{g_1^{-1}}g_1g_2=t([c_{g_1}c_{g_2}])g_1g_2=(g_1g_2)^\tau.\]
Now (b) follows by induction on $n$. 

(c) follows from the definition of $\tau$.
\end{proof}

If $t$ is a 1-cocycle, the function $\tau:\Gamma^*\rightarrow\Gamma^*$ defined in Lemma \ref{inclusion normalized cocycle lem} (b) is called the \emph{rigid map} associated to $t$.

\begin{lem}[{cf. \cite[Lemma 5.2]{GL}}]\label{inclusion normalized cocycle lem2}
Let $\Gamma$ be a virtually discrete $p$-toral group, $S\in\Syl_p(\Gamma)$ and $\Fc=\Fc_S(\Gamma)$.
Let $\Qc\subseteq\Fc^c$ be an $\Fc$-invariant interval such that $S\in\Qc$.
Let $t$ be an inclusion-normalized $1$-cocycle for the functor $\Zc_\Fc^\Qc$ and let $\tau$ be the rigid map associated to $t$.
\begin{enumerate}[(a)]
\item For each $Q\in\Qc$ which is fully normalized in $\Fc$, there is $z\in Z(N_S(Q))$ such that $\tau$ is conjugation by $z$ on $N_\Gamma(Q)$.
\item If $z\in Z(S)$, and $u$ is the constant $0$-cochain defined by $u(Q)=z$ for each $Q\in\Qc$, then $du$ is inclusion-normalized, and the rigid map $v$ associated with $du$ is conjugation by $z$ on $\Gamma^*$.
\item If $\Cc\subseteq\Fc^f$ is a conjugation family for $\Fc$, and $\tau$ is the identity on $N_\Gamma(Q)$ for each $Q\in\Cc\cap\Qc$, then $\tau$ is the identity on $\Gamma^*$.
\end{enumerate}
\end{lem}

\begin{proof}
Let $Q\in\Qc$ be as in (a). From Lemma \ref{inclusion normalized cocycle lem} (b), $\tau$ induces an automorphism of $N_\Gamma(Q)$ that is the identity on $N_S(Q)$. In the proof of \cite[Lemma 1.2]{OV2} it is shown that there is a commutative diagram
\begin{equation}\label{diagram1}
\xymatrix{
1\ar[r] & Z^1(\Out_\Fc(Q);Z(Q))\ar[r]^{\tilde{\eta}}\ar[d] & \Aut(N_\Gamma(Q),Q)\ar[r]\ar[d] & \Aut(Q)\ar[d]\\
1\ar[r] & H^1(\Out_\Fc(Q);Z(Q))\ar[r]^\eta & \Out(N_\Gamma(Q),Q) \ar[r] & \Out(Q)
}
\end{equation}
with exact rows, where $\Aut(N_\Gamma(Q),Q)$, respectively $\Out(N_\Gamma(Q),Q)$  is the subgroup of automorphisms, respectively outer-automorphisms, of $N_\Gamma(Q)$ that leave $Q$ invariant and where $\tilde{\eta}$ maps the restriction of $t$ (to $\Out_\Gamma(Q)$) to the restriction of $\tau$ to $N_\Gamma(Q)$. The restriction map $H^1(\Out_\Fc(Q);Z(Q))\rightarrow H^1(\Out_S(Q);Z(Q))$ is injective since $\Out_S(Q)$ is a Sylow $p$-subgroup of $\Out_\Fc(Q)$ ($Q$ is fully normalized and thus fully automized by \ref{saturation} (I)). Since $t$ is zero on $\Out_S(P)$, $t$ represents the zero class in $H^1(\Out_\Fc(Q);Z(Q))$. It  follows from \eqref{diagram1} that $\tau$ induces an inner automorphism of $N_\Gamma(Q)$. As $Q\in\Fc$ and $\tau$ is the identity on $N_S(Q)$, $\tau$ is conjugation by an element in $Z(N_S(Q))$.

With $u$ as in (b), we see that $du([\iota_P^Q])=u(Q)u(P)^{-1}=zz^{-1}$, for any inclusion $\iota_P^Q$ with $P$ in $\Qc$ (and when $P\not\in\Qc$ by definition of the differential). Moreover, for $g\in\Gamma^*$,
\[g^v=du([c_g])g=z^{-1}z^{g^{-1}}g=g^z\]
Thus (b) holds.

Fix a conjugacy family $\Cc\subseteq\Fc^f$ and assume that $\tau$ is the identity on$N_\Gamma(T)$ for every $T\in\Cc\cap\Qc$. Fix $g\in\Gamma^*$ and choose $Q\in\Qc$ such that $Q^g\leq S$. Then there is $n\geq 1$, subgroups $T_1,T_2,\dots,T_n\in\Cc$, and elements $g_i\in N_\Gamma(T_i)$ such that $g=g_1g_2\cdots g_n$, $Q\leq T_1$ and $Q^{g_1\cdots g_i}\leq T_{i+1}$ for each $1\leq i\leq n$. As $\Qc$ is $\Fc$-invariant and closed with respect to overgroups, $T_i\in\Qc$ for each $i$. Since $\tau$ fixes $g_i$ for each $i$ by assumption, $\tau$ fixes $g$ by Lemma \ref{inclusion normalized cocycle lem}(b).
\end{proof}

\begin{thm}[{cf. \cite[Theorem 5.5]{GL}}]\label{5.4}
Let $(\Gamma,S,Y)$ be a reduced setup for the prime 2. Set $D=Z(Y)$, $\Fc=\Fc_S(\Gamma)$, $G=\Gamma/C_\Gamma(D)$, $\A=\A_D(G)^\circ\cup\hat{\A}_D(G)^\circ$, $\T=\T_D(G)$, $\Bc=\A\smallsetminus \T$ and 
\[\Rc=\{P\in\Sc_p(S)_{\geq Y}\mid J_\A(P)=Y\}.\]
Assume $\Bc\neq \emptyset$. Then, $L^2(\Fc;\Rc)=0$.
\end{thm}

\begin{proof}
If $\Rc=\Sc_p(S)_{\geq Y}$ then,  by Lemma \ref{higherlimitswithS}(b), $L^2(\Fc;\Rc)=0$. Hence we can assume that $\Qc=\Sc_p(S)_{\geq Y}\smallsetminus \Rc\neq \emptyset$. 
By remark \ref{inequalitiesforJ}, $\Qc$ is closed with respect to overgroups, and thus $S\in\Qc$. Moreover, again by remark \ref{inequalitiesforJ}, $J_\A(J_\A(Q))=J_\A(Q)$ for every $Q\leq S$. Therefore $J_\A(Q)\in \Qc$ for every $Q\in\Qc$.

Let $\Gamma^*\subseteq \Gamma$ be the subset of those $g\in\Gamma$ for which there is $Q\in\Qc$ with $Q^g\in \Qc$. Here $\Qc$ and $\Rc$ are $\Fc$-invariant intervals that together satisfy the hypotheses of Lemma \ref{exactseq}. Then, by Lemma \ref{exactseq}(a), it is enough to show that $L^1(\Fc;\Qc)=0$ in order to conclude that $L^2(\Fc;R)=0$.
Fix a 1-cocycle $t$ for the functor $\Zc_\Fc^\Qc$. To show that $t$ is  cohomologous to $0$, we may assume by Lemma \ref{inclusion normalized cocycle lem} that $t$ is inclusion-normalized. Let $\tau:\Gamma^*\rightarrow \Gamma^*$ be the rigid map associated with $t$.

As in the proof of \cite[Theorem 5.5]{GL}, the proof splits into two cases. In Case 1, we assume that a member of $\A$ has order at least 4. In Case 2, we assume that every member of $\A$ has order 2. 

First, fix the following notations. For $X\subseteq \Gamma$, $\overline{X}$ denote its image modulo $C_\Gamma(D)$. If $P$ is a Sylow $2$-subgroup of a subgroup $H\leq\Gamma$, such that every member of $\A\cap \overline{P}$ has order $2$, define $\Bc(P,H)$ to be the collection of subgroups in $\A\cap \overline{P}$ that are not solitary in $\overline{H}$ relative to $\overline{P}$. Let $\Bc(P)$ denote the set of subgroups in $\A\cap \overline{P}$ that are semisolitary relative to $\overline{P}$, and set $\A_{\geq 4}=\{A\in\A\mid|A|\geq4\}$.

Define 
\begin{align*}
J_1(P)&=J_{\A_{\geq4}}(P); \text{ and}\\
J_2(P)&=J_{\A_{\geq4}\cup\Bc}(P).
\end{align*}
Thus $Y\leq J_1(P)\leq J_2(P)\leq J_\A(P)$ whenever $P\geq Y$. We now define mappings $W_1$ and $W_2$ on $\Sc_p(S)$, to be employed in the respective cases. In both cases, set $W_i(P)=P$ if $P$ does not contain $Y$. For $P\geq Y$, set
\[
W_1(P)=\left\lbrace 
\begin{array}{ll}
J_1(P) & \text{if } \A_{\geq 4}\cap \overline{P}\neq\emptyset;\\
J_2(P) & \text{if } \A_{\geq 4}\cap \overline{P}=\emptyset\text{ and }\Bc(P)\neq\emptyset; \text{ and}\\
J_\A(P) & \text{otherwise,}
\end{array}\right.
\]
and
\[
W_2(P)=\left\lbrace
\begin{array}{ll}
J_{\Bc(S,\Gamma)}(S) & \text{whenever } J_{\Bc(S,\Gamma)}(S)\leq P;\text{ and}\\
J_\A(P) &\text{otherwise.} 
\end{array}\right.
\]
In either case, notice that $W_i(P)\in \Qc$, and that $W_i(W_i(P))=W_i(P)$ whenever $P\in \Qc$.

Let $W$ be either $W_1$ or $W_2$, Then $W(S)$ is normal in $S$, so the restriction of $\tau$ to $N_\Gamma(W(S))$ is conjugation by an element $z\in Z(S)$ by Lemma \ref{inclusion normalized cocycle lem2}(a). Upon replacing $t$ by $tdu$ where $u$ is the constant $0$-cochain defined by $u(Q)=z^{-1}$ for each $Q\in\Qc$, and upon replacing  $\tau$ by the rigid map associated with $tdu$, we may assume by Lemma \ref{inclusion normalized cocycle lem2}(b) that

\begin{fact}\label{5.4.1}
$\tau$ is the identity on $N_\Gamma(W(S))$.
\end{fact}

Recall from the definition of $W_1$ above that $W_1(P)\neq 1$ whenever $P\neq 1$. Moreover, all subcollections used in defining $W_1$ are $G$-invariant. Therefore, $W_1$ is an $\Fc$-conjugacy functor in the sense of Definition \ref{conjugacy functor}. If $W=W_2$ and $\A=\A_D(G)_2$, $W$ is an $\Fc$-conjugacy functor as then $\Bc(S,\Gamma)$ is $\Fc$-weakly closed in $S$ by Lemma \ref{4.14}. By Theorem \ref{well placed subgroup thm}, the collection $\Cc$ of subgroups of $S$ which are well-placed with respect to $W$ forms a conjugacy family for $\Fc$. If a subgroup $P$ is well-placed with respect to $W$, then so is $W(P)$ because $W(W(P))=W(P)$.
Set
\[\Wc=\left\lbrace Q\in\Cc\cap\Qc\mid W(Q)=Q\right\rbrace.\]
If $Q\in\Cc\cap\Qc$, then $W(Q)\in\Wc$, and $N_\Gamma(Q)\leq N_\Gamma(W(Q))$ (by Definition \ref{conjugacy functor} of a conjugacy functor). Therefore, by Lemma \ref{inclusion normalized cocycle lem2}(c):
\begin{fact}\label{5.4.2}
If $\tau$ is the identity on $N_\Gamma(Q)$ for each $Q\in\Wc$, then $\tau$ is the identity on $\Gamma^*$.
\end{fact}

For each $Q\in\Wc$, the restriction of $\tau$ to $H=N_\Gamma(Q)$ is conjugation by an element $z_H\in Z(N_S(Q))$. But $Z(N_S(Q))\leq Z(Y)=D$ by Lemma \ref{inclusion normalized cocycle lem2}(a). Hence, if $N_H(W(N_S(Q)))$ fixes $z_H$, and if this normalizer controls fixed points in $H$ on $D$, then $\tau$ is the identity on $H$. This is the key point, and we record it as follows.

\begin{fact}\label{5.4.3}
Assume $Q\in \Wc$ and $H=N_\Gamma(Q)$. If $\tau$ is the identity on $N_\Gamma(W(N_S(Q)))$ and $C_D(H)=C_D(N_\Gamma(W(N_S(Q))))$ then $\tau$ is the identity on $H$.
\end{fact} 

We now distinguish between the two cases. In each, we prove that $\tau$ is the identity on $N_\Gamma(Q)$ for each $Q\in\Wc$ by induction on the index of $N_S(Q)$ in $S$. Notice that this index is always finite since $Q\geq Y$ and $Y\geq \Gamma_0$ by definition of a general setup. Fix $Q\in\Qc$ and set $S^*=N_S(Q)$ and $H=N_\Gamma(Q)$ for short.

The Norm Arguments in Section \ref{normargp=2 section} are statements about control of fixed points in $G$. Each member of $\Qc$ contains $Y=C_S(D)$ by definition and so, Lemma \ref{barnorm} provides the transition from control of fixed points by normalizers within $G$ and those within $\Gamma$. We apply Lemma \ref{barnorm} implicitly for this transition in the arguments that follows.

\begin{case}\label{case 1}
Some member of $\A$ has order at least 4.
\end{case}

Put $W=W_1$. Assume first that $S^*=S$. Since Case \ref{case 1} holds, the collection $\A_{\geq 4}:=\A_D(G)^\circ_{\geq 4}\cup \hat{\A}_D(G)^\circ$ is nonempty. Hence, $W(S^*)=J(S)$ by definition of $W$. By Lemma \ref{4.6}, \eqref{4.2} is satisfied with $(\overline{H},\overline{S},D,\A_{\geq 4}\cap \overline{H},N_{\overline{H}}(\overline{J_1(S)}))$ in the role of the five-tuple $(G,S,D,\A,H)$ of that lemma. Hence, $C_D(H)=C_D(N_H(J_1(S)))$ by Corollary \ref{normargumentp=2}. Furthermore, $\tau$ is the identity on $N_\Gamma(W(S))=N_\Gamma(J_1(S))$ by \ref{5.4.1}, so that $\tau$ is the identity on $H$ by \ref{5.4.3}.

Now assume that $S^*<S$. If $\A_{\geq 4}\cap \overline{S^*}\neq \emptyset$ (that is $J_1(S^*)>Y$), then $W(S^*)=J_1(S^*)$, and $C_D(H)=C_D(N_H(W(S^*)))$ by Lemma \ref{4.6} and Corollary \ref{normargumentp=2} as before. By Lemma \ref{conj functor lem}(c), $S^*<N_S(W(S^*))$ and thus $\tau$ is the identity on $N_\Gamma(W(S^*)))$ by induction. Thus $\tau$ is the identity on $H$ by \ref{5.4.3}.

Assume for the remainder of Case \ref{case 1} that $\A_{\geq 4}\cap \overline{S^*}$ is empty. In particular, $\hat{\A}_D(G)\cap \overline{S^*}$ is empty. However, $\A\cap \overline{S^*}$ is nonempty since $S^*> Q\in \Qc$ (if $Q\in Q$ , $J_\A(Q)> Y$), and every member of $\A\cap \overline{S^*}$ is of order 2. Moreover, 
\begin{equation}\label{5.5}
Q=W(Q)\leq J_\A(Q)\leq J_A(S^*)\qquad \text{ and }\qquad \overline{J_\A(S^*)}\,\text{ is elementary abelian}.
\end{equation}
from Lemma \ref{4.8}(c).

Assume next that $\Bc(S^*)$ is nonempty. Then $W(S^*)=J_2(S^*)>Y$. Since no member of $\Bc(S^*)$ is solitary in $\overline{H}$ relative to $\overline{S^*}$ by Remark \ref{4.13}, we see that $\Bc(S^*)\subseteq \Bc(S^*,H)$, and that
\[
W(S^*)=J_2(S^*)=J_{\Bc(S^*)}(S^*)\leq J_{\Bc(S^*,H)}(S^*).
\]
Since, by \eqref{5.5}, $\overline{J_\A(S^*)}$ is elementary abelian  and the collection of semisolitary subgroups relative to $\overline{S^*}$ is invariant under conjugation in $N_{\overline{H}}(\overline{J_\A(S^*)})$, \cite[Lemma 8.1]{GL} and Lemma \ref{barnorm} shows that $W(S^*)$ is weakly closed in $S^*$ with respect to $H$. Apply Lemma \ref{4.15} with $(\overline{H},\overline{S^*},D,\A\cap\overline{S^*},\overline{W(S^*)})$ in the role of the five-tuple $(G,S,D,\A,J^*)$ of that Lemma to obtain the hypotheses of Corollary \ref{normargumentp=2}, which gives $C_D(H)=C_D(N_H(W(S^*)))$ as before. But, since $S^*<S$, we have $S^*<N_S(W(S^*))$ from Lemma \ref{conj functor lem}(c). As $\tau$ is the identity on $N_\Gamma(W(S^*))$ by induction, we conclude by \ref{5.4.3} that $\tau$ is the identity on $H$.

Finally, assume that $\Bc(S^*)$ is empty. Then $W(S^*)=J_\A(S^*)$ by the definition of $W$, and every element of $\A\cap \overline{S^*}$ is semisolitary relative to $\overline{S^*}$. We will show that this leads to a contradiction. Since Case \ref{case 1} holds, $J_\A(S^*)<J_\A(S)$. Now $J_\A(S^*)<J_\A(N_S( J_\A(S^*)))$ (and by Lemma \ref{conj functor lem}(b)) and so, there exists $A\leq S$ with $\overline{A}\in\A\cap\overline{S}$ such that
\begin{equation}\label{5.6}
A\leq J_\A(N_S( J_\A(S^*))) \qquad \text{and}\qquad A\not\leq J_\A(S^*).
\end{equation}
It follows from the definitions that each subgroup of $S^*$ that is semisolitary relative to $S^*$ is also semisolitary relative to $J_\A(S^*)$. As $A$ normalizes $J_\A(S^*)$, we see that $\overline{A}$ permutes the elements of $\A\cap \overline{S^*}=\A\cap \overline{J_\A(S^*)}$ by conjugation. We are thus in the situation of Lemma \ref{4.29} with $\overline{J_\A(S^*)}$ in the role of $P$, $\overline{A}$ in the role of $A$ and $\A\cap\overline{S^*}$ in the role of $\T$ there. By that lemma, $\overline{A}$ normalizes every member of $\A\cap S^*$. Thus A normalizes each of their preimages in $S$. However, $Q$ is generated by the preimage of a subset of $\A\cap \overline{S^*}$ by \eqref{5.5}, where $W(Q)=J_\A(Q)$ in the current situation. Hence, $A$ normalizes $Q$. But then $A\leq J_\A(N_S(Q))=J_\A(S^*)$, and we have a contradiction to \eqref{5.6}. This completes the proof in Case \ref{case 1}.

\begin{case}\label{case 2}
Each member of $\A$ is of order 2.
\end{case}

Put $W=W_2$ and assume first that $S^*=S$. By assumption, $\A=\A_D(G)_2$, and so $\hat{A}_D(G)=\emptyset$ by Remark \ref{4.5}. Set $J=J_{\Bc(S,H)}(S)$ and note that $W(S)=J_{\Bc(S,\Gamma)}(S)$ by the definition of $W_2$. By Definition \ref{solitary}, every element of $\A$ that is solitary in $\overline{H}$ relative to $\overline{S}$ is also solitary in $G$ relative to $\overline{S}$. Thus $\Bc(S,\Gamma)\subseteq \Bc(S,H)$, and 
\[W(S)\leq J\leq J_\A(S).\]
We saw earlier (after \ref{5.4.1}) that $W(S)$ is weakly closed in $S$ with respect to $\Gamma$ whenever $\A=\A_D(G)_2$, as holds in the present case. So we may apply Lemma \ref{4.15} with $(\overline{H},\overline{S},D,\A\cap \overline{H},\overline{W(S)})$ in the role of the five-tuple $(G,S,D,\A,J^*)$ of that lemma to obtain the hypotheses of Corollary \ref{normargumentp=2}. This yields $C_D(H)=C_D(N_H(W(S)))$. Furthermore, $\tau$ is the identity on $N_\Gamma(W(S))$ by \ref{5.4.1}, so that $\tau$ is the identity on $H$ by \ref{5.4.3}.

Assume now that $S^*<S$. As $\overline{J_\A(S)}$ is elementary abelian by Lemma \ref{4.8} and $Q\in\Wc$,
\[Q=W(Q)\leq J_\A(Q)\leq J_\A(S),\]
and $\overline{J_\A(S)}$ centralizes $\overline{Q}$. Hence $J_\A(S)\leq S^*$ so that
\[J_\A(S)=J_\A(S^*).\]
Since $J_{\Bc(S,\Gamma)}(S)\leq J_\A(S)$, this shows that $W(S^*)=J_{\Bc(S,\Gamma)}(S)=W(S)$. As in the situation where $Q$ is normal in $S$, we have $\Bc(S^*,\Gamma)\subseteq\Bc(S^*,H)$ and $W(S^*)$ is weakly closed in $S^*$ with respect to $H$. Apply Lemma \ref{4.15} with $(\overline{H},\overline{S^*},D,\A\cap H,\overline{W(S^*)})$ in the role of the five-tuple $(D,S,D,\A,J^*)$ of that lemma to obtain the hypotheses of Corollary \ref{normargumentp=2}. this yields $C_D(H)=C_D(N_H(W(S^*)))$ as before. Furthermore, since $W(S^*)=W(S)$, $\tau$ is the identity on $N_\Gamma(W(S^*))$ by \ref{5.4.2}, so that $\tau$ is the identity on $H$ by \ref{5.4.3}. This concludes the proof in Case \ref{case 2}.
\end{proof}

\subsection{The transvection case}
In this section we give in Proposition \ref{6.9} a proof of \cite[Proposition 3.5]{LL} for $p=2$.

We first give a classification of finite groups with no nontrivial normal $2$-subgroups which act faithfully on a abelian discrete $2$-toral group and are generated by solitary offenders. The result is a consequence, using Lemma \ref{lem solitary}, of the classification given by Glauberman and Lynd \cite[Theorem 6.2]{GL} that was recalled here in Theorem \ref{6.2finite}. The key of the proof of Theorem \ref{6.2finite} is McLaughlin's classification of irreducible subgroups of $SL_n(2)$ generated by transvections, which stands independently from the classification of the finite simple groups.
 
Here, by a \emph{natural} $S_m$-module ($m\geq 3$), we mean the non-trivial composition factor of the standard permutation module of $S_m$ over the field with two elements. 

\begin{thm}[{\cite[Theorem 6.2]{GL}}]\label{6.2finite}
Let $G$ be a finite group, $S\in\Syl_2(G)$ and let $D$ be an abelian $2$-group on which $G$ acts faithfully. Set $\T=\T_D(G)$. Assume that $O_2(G)=1$ and that $G=\langle \T\rangle$ is generated by its solitary offenders. Then there exist a positive integer $r$ and subgroups $E_1,E_2,\dots,E_r$ of $G$ such that,
\begin{enumerate}[(a)]
\item $G=E_1\times \cdots\times E_r$, $\T=(\T\cap E_1)\cup\cdots\cup(\T\cap E_r)$, and $E_i\cong S_{m_i}$ with $m_i$ odd for each $i$; and
\item $D=V_1\times\dots\times V_r\times C_D(G)$ with $V_i=[D,E_i]$ a natural $S_{m_i}$-module, and with $[V_i,E_j]=1$ for $j\neq 1$. 
\end{enumerate}
\end{thm}

\begin{cor}\label{6.2}
Let $G$ be a finite group, $S\in\Syl_2(G)$ and let $D$ be an abelian discrete $2$-toral group on which $G$ acts faithfully. Set $\T=\T_D(G)$. Assume that $O_2(G)=1$ and that $G=\langle \T\rangle$ is generated by its solitary offenders. Then there exist a positive integer $r$ and subgroups $E_1,E_2,\dots,E_r$ of $G$ such that,
\begin{enumerate}[(a)]
\item $G=E_1\times \cdots\times E_r$, $\T=(\T\cap E_1)\cup\cdots\cup(\T\cap E_r)$, and $E_i\cong S_{m_i}$ with $m_i$ odd for each $i$; and
\item $D=V_1\times\dots\times V_r\times C_D(G)$ with $V_i=[D,E_i]$ a natural $S_{m_i}$-module, and with $[V_i,E_j]=1$ for $j\neq 1$. 
\end{enumerate}
\end{cor}

\begin{proof}
By Lemma \ref{action on discrete=action on finite} and Lemma \ref{lem solitary}, if $n$ is large enough we have $D=D_0\Omega_n(D)$ and $\T_D(G)=\T_{\Omega_n(D)}(G)$. In particular, under the hypothesis of Corollary \ref{6.2}, we are under the hypothesis of Theorem \ref{6.2finite} with $\Omega_n(D)$ in the role of $D$.
Thus, by Theorem \ref{6.2finite}, $G=E_1\times \cdots\times E_r$, $\T=(\T\cap E_1)\cup\cdots\cup(\T\cap E_r)$, and $E_i\cong S_{m_i}$ with $m_i$ odd for each $i$. We also have by Theorem \ref{6.2finite} that $\Omega_n(D)=V_1\times\dots\times V_r\times C_{\Omega_n(D)}(G)$ where $V_i=[\Omega_n(D),E_i]$ is a natural $S_{m_i}$-module, and with $[V_i,E_j]=1$ for $j\neq 1$.  
Moreover, as $D_0\leq C_D(T)$ for every $T\in\T$ and $G=\langle\T\rangle$, $D_0\leq C_D(G)$ and, for every $i$, $V_i=[\Omega_n(D),E_i]=[D,E_i]$. Thus the corollary holds.
\end{proof}

\begin{prop}[{cf. \cite[Proposition 6.9]{GL}}]\label{6.9}
Let $(\Gamma,S,Y)$ be a general setup for the prime 2. Set $\Fc=\Fc_S(\Gamma)$, $D=Z(Y)$, and $G=\Gamma/C_\Gamma(D)$. Let $\Rc\subseteq \Sc_p(S)_{\geq Y}$ be an $\Fc$-invariant interval such that for each $Q\in\Sc_p(S)_{\geq Y}$, $Q\in\Rc$ if and only if $ J_{\A_D(G)}(Q)=R$. Then $L^k(\Fc;R)=0$ for all $k\geq 2$.
\end{prop}

\begin{proof}
Assume that the conclusion is false. Let $(\Gamma,S,Y,\Rc,k)$ be a counterexample for which the four-tuple $(k,\ord(\Gamma),|\Gamma/Y|,|\Rc|)$ is minimal in the left lexicographic ordering. Steps 1-3 in the proof of \cite[Proposition 3.5]{LL} show that $\Rc=\{R\in \Sc_p(S)_{\geq Y}\leq S\mid J_{\A_D(G)}(R)=Y\}$, $k=2$, and $(\Gamma,S,Y)$ is a reduced setup.

Let $\Qc=\Sc_p(S)_{\geq Y}\smallsetminus \Rc$ and $\A=\A_D(G)^\circ\cup \hat{\A}_D(G)^\circ$. Since $(\Gamma,S,Y)$ is a counter-example, Theorem \ref{5.4} shows:
\begin{fact}\label{6.9.1}
$\A=\T_D(G)$.
\end{fact}

That is, $\hat{\A}_D(G)$ is empty and every best offender minimal under inclusion is solitary in $G$ relative to $\overline{S}$. Since $L^2(\Fc;\Rc)\neq 0$, Lemma \ref{exactseq} implies:

\begin{fact}\label{6.9.2}
$L^1(\Fc;\Qc)\neq 0$.
\end{fact}

We prove next that 

\begin{fact}\label{6.9.3}
$G=\langle\A\rangle$.
\end{fact}
\begin{proof}
Let $G_0=\langle\A\rangle$. Let $\Gamma_0$ be the preimage of $G_0$ in $\Gamma$, set $S_0=S\cap\Gamma_0$, $\Fc_0=\Fc_{S_0}(\Gamma_0)$ and $\Qc_0=\Sc_p(S)_{\geq Y}\cap\Qc$. Then $\Gamma_0\unlhd\Gamma$ and $Y\leq S_0$. Furthermore, $(\Gamma_0,Y,S_0)$ is a reduced setup and $\Qc_0$ is an $\Fc_0$-invariant interval. Since each member of $\A$ is contained in $G_0$, we have $\Gamma_0\cap Q\in\Qc$ for each $Q\in\Qc$. By Lemma \ref{2.7}, the restriction map induces an injection $L^1(\Fc,\Qc)\rightarrow L^1(\Fc_0;\Qc_0)$ and so $L^1(\Fc_0;\Qc_0)\neq 0$ by \ref{6.9.2}. Hence $\Gamma=\Gamma_0$ by minimality of $|\Gamma|$.
\end{proof}

By \ref{6.9.3}, we are under the hypothesis of Corollary \ref{6.2}, and $G$ and its action on $D$ are described by Corollary \ref{6.2}. We adopt the notation in that corollary for the remainder of the proof. In the decomposition given by \ref{6.2}(b) each $V_i$ is $G$-invariant, and so each $V_i$ is $S$-invariant. Thus the centralizer of $S$ in $D$ factors as

\begin{fact}\label{6.9.4}
$C_D(S)=C_{V_1}(S)\times C_{V_2}(S)\times\cdots\times C_{V_r}(S)\times C_D(G).$
\end{fact}  

Fix an inclusion-normalized 1-cocyle $t$ for $\Zc_\Fc^\Qc$ representing a non-zero class in $L^1(\Fc;\Qc)$ (which is nontrivial by \ref{6.9.2}), and let $\tau:\Gamma^*\rightarrow\Gamma^*$ be the rigid map associated with $t$. We show that
\begin{fact}\label{6.9.5}
$r=1$.
\end{fact}

\begin{proof}
Assume $r>1$. Set $G_1=E_1$ and $G_2=E_2\times E_3\times\cdots\times E_r$.
For $i=1,2$, let $K_i$ be the preimage of $G_i$ in $\Gamma$ and set $\Gamma_i=K_iS$ and $\Fc_i=\Fc_S(\Gamma_i)$. We have $\Gamma_i<\Gamma$ and $(\Gamma_i,S,Y)$ is a general setup by assumption. Hence $L^1(\Fc_i;\Qc)\cong L^2(\Fc_i;\Rc)=0$ by minimality of $\ord(\Gamma)$. As the restriction of $t$ to $\Or(\Fc_i^c)$ represents the zero class and is inclusion-normalized, and since $S\in\Qc$, the restriction of $t$ to $\Or(\Fc_i^c)$ is of the form $du$ with $u$ a constant 0-cochain.  It  follows from this and from Lemma \ref{inclusion normalized cocycle lem2}(b) that there are elements $z_i\in Z(S)=C_D(S)$ 
such that $\tau$ is the conjugation by $z_i$ on $\Gamma_i^*$. By \ref{6.9.4} we can in fact choose $z_i$ such that $z_i\in [D,\Gamma_i]\leq C_D(\Gamma_{3-i})$. Set $t'=tdu$ where $u$ is the constant $0$-cochain defined by $u(P)=(z_1z_2)^{-1}$ for each $P\in\Qc$. Then by Lemma \ref{inclusion normalized cocycle lem}(b), upon replacing $t$ by $t'$ and $\tau$ by $\tau'$, we may assume that $\tau$ is the identity when restricted to $\Gamma_i^*$, for each $i=1,2$.

The objective is now to show that $\tau$ is the identity on $\Gamma^*$. Let $W$ be the $\Fc$-conjugacy functor defined by $W(P)=J_\A(P)$ for each $P\geq Y$ and by $W(P)=P$ otherwise. Let $\Cc$ be the collection of subgroups $P$ of $S$ such that $W(P)=P$ and $P$ is well-placed with respect to $W$. By Theorem \ref{well placed subgroup thm} the family of well-placed subgroups of $S$ is a conjugacy family for $\Fc$. By Lemma \ref{conj functor lem}(a), we have $N_\Gamma(P)\leq N_\Gamma(W(P))$ for each $P\leq S$. Then, by Definition of $J_\A(P)$, $W(W(P))=W(P)$. Hence the collection $\Cc$ is also a conjugacy family for $\Fc$. In particular, as $t$ is inclusion-normalized, by Lemma 5.2(c), it now suffices to show that $\tau$ is the identity on $N_\Gamma(Q)$ for each $Q\in\Cc\cap \Qc$.
 
Let $Q\in\Cc\cap\Qc$. Thus $Q=W(Q)=J_\A(Q)$ and the image $\overline{Q}$ of $Q$ in $G$ is generated by members of $\A$. It follows from Corollary \ref{6.2}(a) that $Q=Q_1Q_2$ with $Q_1\cap Q_2=Y$, where $\overline{Q_1}$ and $\overline{Q_2}$ are the projections in $G_1$ and $G_2$ of $\overline{Q}$. If it happens that $Q_i\not\in\Qc$ for $i=1$ or $2$, this means that $Q_i\in\Rc$ and thus, as $Q= J_\A(Q)$, $Q_i=Y$. Since $Y\not\in \Qc$ but $Q\in\Qc$ we may assume without loss of generality that $Q_2\in\Qc$. As $Q$ is well-placed and $W(Q)=Q$, $Q$ is fully normalized in $\Fc$.

Let $g_2\in N_\Gamma(Q_2)$, and write $g=g_1g_2$ with $g_i\in K_i$. As $(\Gamma,S,Y)$ is a reduced setup, $C_S(D)=Y$. Also, $Y\leq Q_2\leq N_\Gamma(Q)$, thus, by Lemma \ref{barnorm}, $\overline{N_\Gamma(Q_2)}=N_G(\overline{Q_2})=G_1\times N_{G_2}(\overline{Q_2})$. Since $\overline{g_1}\in G_1$ centralizes $\overline{Q_2}$, we may write $g_1=h_1c_1$ where $h_1\in C_{\Gamma_1}(Q_2/Y)\leq N_{\Gamma_1}(Q_2)$ and $c_1\in C_\Gamma(D)$. So, as $h_1\in(\Gamma_1)^*$ and $c_1g_2\in(\Gamma_2)^*$ (both sends $Q_2$ to $Q_2$), we see that $\tau$ is the identity on $N_\Gamma(Q)$ since $N_\Gamma(Q)\leq N_\Gamma(Q_2)$. We conclude that $\tau$ is the identity on $\Gamma^*$, which is a contradiction.
\end{proof}

By \ref{6.9.5}, we may fix $m=2n+1$ so that $G=E_1\cong S_m$. Let $\Omega$ be the set of even-order subsets of $\{1,2,\dots,m\}$. Identify $G$ with $S_m$ and $V:=V_1$ with $\Omega$. We may assume that $S$ stabilizes the collection $\left\lbrace\{2i-1,2i\}\mid 1\leq i\leq n\right\rbrace$. Then $\A\cap \overline{S}=\{\langle(1,2)\rangle,\langle(2,3)\rangle,\cdots, \langle(2n-1,2n)\rangle\}$ and thus:
\begin{equation}\label{Gamma^* equivalent def}
g\in\Gamma^*\qquad\text{ if and only if}\qquad\text{ there is }1\leq i, j \leq n\text{ such that }(2i-1,2i)^{\overline{g}}=(2j-1,2j) 
\end{equation}
where $\overline{g}$ is the image of $g$ in $G$.

Let $A\in\Sc_p(S)_{\geq Y}$ be such that the image $\overline{A}$ of $A$ in $G$ is $\langle (1,2),(3,4),\dots,(2n-1,2n)\rangle$. Set $H=N_\Gamma(A)$. Since $C_S(V)=Y\leq A$, by Lemma \ref{barnorm}, $\overline{H}=N_G(\overline{A})$ is isomorphic to $A\rtimes S_n$. Let also $A_1\in\Sc_p(S)_{\geq Y}$ be such that $\overline{A_i}=\langle (2i-1,2i)\rangle$ and let $K_i=N_\Gamma(A_i)$. As before, since $C_S(V)=Y\leq A$, we have $\overline{K_i}=N_G(\langle (2i-1,2i)\rangle)$. Here $A_1^\Fc=\{A_1,A_2,\dots,A_n\}=A_1^{\Fc_S(H)}$. Thus, replacing $S$ by an $H$-conjugate, we can assume that $A_1$ is fully normalized. We next show:

\begin{fact}\label{Gamma^*}
$\Gamma^*=HK_1H$. 
\end{fact}
\begin{proof}
Indeed, for each $h_1kh_2\in HK_1H$, $\overline{h_1},\overline{h_2}$ acts by conjugation on $\{(1,2),(3,4),\cdots,(2n,2n-1)\}$, and thus, if we write $(2i-1,2i)=(1,2)^{\overline{h_1}^{-1}}$ and $(2j,2j-1)=(1,2)^{\overline{h_2}}$, $(2i-1,2i)^{\overline{h_1kh_2}}=(2j-1,2j)$ and, by \eqref{Gamma^* equivalent def}, $h_1kh_2\in\Gamma^*$. 
Conversely, if $g\in \Gamma^*$, by \eqref{Gamma^* equivalent def}, there are $1\leq i,j\leq n$ such that $(2i-1,2i)^{\overline{g}}=(2j-1,2j)$. Then, if we take $h_1,h_2\in H$ such that $(2i-1,2i)=(1,2)^{\overline{h_1}^{-1}}$ and $(2j,2j-1)=(1,2)^{\overline{h_2}}$, we have $(1,2)^{\overline{h_1^{-1}gh_2^{-1}}}=(1,2)$ which implies that $h_1^{-1}gh_2^{-1}\in K_1$ and $h_1(h_1^{-1}gh_2^{-1})h_2\in HK_1H$.
\end{proof}

Now, in order to prove that $\tau$ is the identity on $\Gamma^*$, it is enough to show that $\tau$ is the identity on $H$ and $K_1$. First of all, as $A\in\Qc\cap\Fc^f$, by \ref{inclusion normalized cocycle lem2}(a) and as $A\unlhd S$, there is $z\in Z(S)$ such that the restriction of $\tau$ to $H$ is the conjugation by $z$. Hence, by Lemma \ref{inclusion normalized cocycle lem2}(b), we may adjust $t$ by a coboundary to get that $\tau$ is the identity on $H$.

We also know that $A_1\in\Qc\cap\Fc^f$ and then, by Lemma \ref{inclusion normalized cocycle lem2}(a), the restriction of $\tau$ to $K_1$ is the conjugation by an element $z$ in $V$ (recall that $D=V\times C_D(\Gamma)$). Remark also that $\tau$ is the identity on $K_1\cap H$ which implies that $z\in C_D(K_1\cap H)\subset$.
Now, 
\[\overline{K_1\cap H}=\langle(1,2)\rangle\times \left(\langle(3,4),\dots,(2n-1,2n)\rangle\rtimes S_{n-1})\right.\]
Thus, if we write $z_1=\{1,2\}\in V$ ,$z_1'=\{3,4,\dots 2n\}\in V$ and $z=\{1,2,\dots,2n\}=z_1z_1'\in V$, 
\[C_V(K_1\cap H)=\langle z_1,z_1'\rangle.\]
As $z\in C_V(H)$ and $z_1\in C_V(K_1)$, $c_{z}=c_{z_1'}$ on $K_1$ and we can, if necessary, change $t$ by the coboundary $du$ where $u$ is the constant 0-cochain defined by $u(P)=z$ for $P\in \Qc$, to get that $\tau$ is the identity on $H$ and $K$ and thus on $\Gamma^*$ by \ref{Gamma^*}.

In the end, we proved that if $t$ is an inclusion-normalized 1-cocycle for $\Zc_\Fc^\Qc$ representing a non-zero class in $H^1(\Fc;\Qc)$ then the associated rigid map $\tau:\Gamma^*\rightarrow\Gamma^*$ is the identity. This is a contradiction in view of Lemma \ref{inclusion normalized cocycle lem}(c).  
\end{proof}

%
%

\bibliography{biblio}{}
\bibliographystyle{abstract}

%
%

\end{document}